\documentclass[reqno,12pt]{amsart}
\usepackage{amsmath, latexsym, amsfonts, amssymb, amsthm, amscd}
\usepackage{mathrsfs,enumerate}

\usepackage{color}

\numberwithin{equation}{section}

\newtheorem{theorem}{Theorem}[section]
\newtheorem{lemma}[theorem]{Lemma}
\newtheorem{prop}[theorem]{Proposition}

\newtheorem{rem}[theorem]{Remark}
\newtheorem{definition}[theorem]{Definition}
\newtheorem{example}[theorem]{Example}

\newtheorem{remark}[theorem]{Remark}

\newcommand{\SoC}{{\mathcal C}}

\newcommand{\kk}{\mathbf k}
\newcommand{\xx}{\mathbf x}
\newcommand{\KK}{\mathbf K}
\newcommand{\XX}{\mathbf X}

\newcommand{\gep}{\varepsilon}       % \ge already exists...

\newcommand{\cM}{{\ensuremath{\mathcal M}} }

\newcommand{\cX}{{\ensuremath{\mathcal X}} }

\newcommand{\cZ}{{\ensuremath{\mathcal Z}} }

\newcommand{\E}{{\ensuremath{\mathbb E}} }

\newcommand{\N}{{\ensuremath{\mathbb N}} }

\newcommand{\bbP}{{\ensuremath{\mathbb P}} }

\newcommand{\R}{{\ensuremath{\mathbb R}} }

\newcommand{\z}{{\ensuremath{\mathcal Z}} }

\newfont{\indic}{bbmss12}

\def\un#1{\hbox{{\indic 1}$_{#1}$}}

\newcommand\MI{\operatorname{MI}}
\newcommand\I{\mathcal I}
\renewcommand\H{\operatorname{H}}

\newcommand\eps{\epsilon}
\newcommand\EX{\operatorname{EX}}

\begin{document}

\title[Neural Complexity]{A probabilistic study of neural complexity}

\author{J. Buzzi}

\address{Laboratoire de Math\'ematique d'Orsay - C.N.R.S.  (U.M.R. 8628) \& Universit\'e Paris-Sud
\\
Universit\'e Paris-Sud, F-91405 Orsay Cedex, France}
\email{jerome.buzzi\@@math.u-psud.fr}

\author{L. Zambotti}

\address{Laboratoire de Probabilit{\'e}s et Mod\`eles Al\'eatoires (CNRS U.M.R. 7599) and  Universit{\'e} Paris 6
-- Pierre et Marie Curie, U.F.R. Mathematiques, Case 188, 4 place
Jussieu, 75252 Paris cedex 05, France }
\email{lorenzo.zambotti\@@upmc.fr}

\begin{abstract}
G. Edelman, O. Sporns, and G. Tononi have introduced 
 the \emph{neural complexity} of a family of random variables,
defining it as a specific average of mutual information
over subfamilies. We show that their choice of weights satisfies
two natural properties, namely exchangeability and additivity, and
we call any functional satisfying these two properties an
\emph{intricacy}. We classify all intricacies in terms of probability laws on the unit
interval and study the growth rate of maximal intricacies when the
size of the system goes to infinity.
For systems of a fixed size, we show that maximizers have small support
and exchangeable systems have small intricacy. In particular,
maximizing intricacy leads to spontaneous symmetry breaking
and failure of uniqueness.
\end{abstract}



\maketitle

%\tableofcontents

%%%%%%%%%%%%%%%%%%%%%%%%%%%%%%%%%%%%%%%%%%%%%%%%%%%%%%%%%%%%%%%%%%%%%%%%%%%%%%%%%%
%%%%%%%%%%%%%%%%%%%%%%%%%%%%%%%%%%%%%%%%%%%%%%%%%%%%%%%%%%%%%%%%%%%%%%%%%%%%%%%%%%
\section{Introduction}
%%%%%%%%%%%%%%%%%%%%%%%%%%%%%%%%%%%%%%%%%%%%%%%%%%%%%%%%%%%%%%%%%%%%%%%%%%%%%%%%%%
%%%%%%%%%%%%%%%%%%%%%%%%%%%%%%%%%%%%%%%%%%%%%%%%%%%%%%%%%%%%%%%%%%%%%%%%%%%%%%%%%%

\subsection{A functional over random systems}
Natural sciences have to deal with "complex systems" in some obvious and not
so obvious meanings. Such notions first appeared in thermodynamics. Entropy 
is now recognized as the fundamental measure of complexity in the sense of 
randomness and it is playing a key role as well in information theory, probability
and dynamics \cite{Entropy}. Much more recently, subtler forms of complexity have been considered 
in various physical problems \cite{Bak95,Bennett90,Crutchfield,Goldenfeld99},
though there does not seem to be a single satisfactory measure yet.

Related questions also arise in biology. In their study of high-level neural networks,
  G. Edelman, O. Sporns and G. Tononi have
argued that the relevant complexity
should be a combination of high {\it integration} and high {\it
differentiation}. In \cite{Tononi94} they have introduced a
quantitative measure of this kind of complexity under the name of
{\it neural complexity}. As we shall see, this concept is
strikingly general and has interesting mathematical properties.

In the biological \cite{Edelman01,Holthausen99,
Krichmar04,Seth07,Shanahan08,Sporns00,Sporns02,Sporns07,Tononi96,
Tononi99} and physical \cite{Barnett09,DeLuca04} literature, several authors
have used numerical experiments based on Gaussian approximations and simple
examples to suggest that high values of this  neural complexity
are indeed associated with non-trivial organization of the network,
away both from complete disorder (maximal entropy and independence of the neurons)
and complete order (zero entropy, i.e., complete determinacy).

The aim of this paper is to provide a mathematical foundation
for the Edelman-Sporns-Tononi complexity.
Indeed, it turns out to belong to a natural class
of functionals: the \emph{averages of mutual informations satisfying
exchangeability and  weak-additivity} (see below and the Appendix
for the needed facts of information theory). The former property
means that the functional is invariant under permutations of
the system. The latter that it is additive over independent
systems. We call these functionals {\it intricacies} and give a
unified probabilistic representation of them.

One of the main thrusts of the above-mentioned work is to understand
how systems with large neural complexity look like. From a mathematical
point of view, this translates into the study of the maximization of such
functionals (under appropriate constraints).

This maximization problem is interesting because of the trade-off between
high entropy and strong dependence which are both required for large
mutual information.  Such \emph{frustration} occurs in spin glass
theory \cite{Tala} and leads to asymmetric and non-unique maximizers.
However, contrarily to that problem, our functional is completely
deterministic and the symmetry breaking (in the language of theoretical
physics) occurs in the maximization itself: we show that the maximizers are
not exchangeable although the functional is.
We also estimate the
growth of the maximal intricacy of finite systems with size
going to infinity and the size of the support of maximizers.

%Finally we extend these considerations to stationary processes.
%In particular, in opposition to exchangeable systems (which have
%small intricacy), the mean intricacy of a stationary process
%can approximate the maximal value found for finite size.

The computation of the exact growth rate of the intricacy as
a function of the size and the analysis of systems with almost
maximal intricacies build on the techniques of this paper,
especially the probabilistic representation below, but require
additional ideas, so are deferred to another paper \cite{BZ2}.

\subsection{Intricacy} We recall that the {\it entropy} of a random variable
$X$ taking values in a finite or countable space $E$ is defined by
\[
\H(X) := -\sum_{x\in E} P_X(x) \, \log(P_X(x)), \qquad
P_X(x):=\bbP(X=x).
\]
Given two discrete random variables defined over the same probability space,
the {\it mutual information} between $X$ and $Y$ is
\[
\MI(X,Y) := \H(X)+\H(Y)-\H(X,Y).
\]
We refer to the appendix for a review of the main properties of the
entropy and the mutual information and to \cite{InfoTheory}
and \cite{Entropy} for introductions to information theory and to
the various roles of entropy in mathematical physics, respectively.
For now, it suffices to recall that
$\MI(X,Y)\geq 0$ is equal to zero if and only if $X$ and $Y$ are independent,
and therefore $\MI(X,Y)$ is a measure of the dependence between $X$ and $Y$.

Edelman, Sporns and Tononi \cite{Tononi94}
consider systems formed by a finite family
$X=(X_i)_{i\in I}$ of random variables and define the following concept of complexity. For any
$S\subset I$, they divide the system in two families
\[
X_S:=(X_i, i\in S), \qquad X_{S^c}:=(X_i, i\in S^c),
\]
where $S^c:=I\backslash S$. Then they compute the mutual informations
$\MI(X_S,X_{S^c})$ and consider an average of these:
\begin{equation}\label{est}
\I(X):=\frac1{|I|+1} \sum_{S\subset I} \frac1{\binom{|I|}{|S|}} \,
\MI(X_S,X_{S^c}),
\end{equation}
where $|I|$ denotes the cardinality of $I$ and $\binom{n}{k}$ is the
binomial coefficient. Note that $\I(X)$ is really a function of
the \emph{law} of $X$ and not of its random values.

The above formula can be read as the expectation of the mutual information
between a random subsystem $X_S$ and its complement $X_{S^c}$
where one chooses uniformly the size $k\in\{0,\ldots,|I|\}$ and then
a subset $S\subset I$ of size $|S|=k$.

In this paper we prove that $\I$ fits into a natural
class of functionals, which we call {\bf intricacies}.
We shall see that these functionals have very
similar, though not identical properties and admit a natural and technically
very useful probabilistic representation by means of a probability measure on $[0,1]$.

Notice that $\I\geq 0$ and $\I=0$ if and only if the system
is an independent family (see Lemma \ref{bert2} below).
In particular, both complete order (a deterministic family $X$) and
total disorder  (an independent family) imply
that every mutual information vanishes and therefore $\I(X)=0$.

On the other hand, to make \eqref{est} large, $X$ must simultaneously display
two different behaviors: a non-trivial correlation between its subsytems and
a large number of internal degrees of freedom. This is the hallmark of
complexity according to Edelman, Sporns and Tononi. The need to
strike a balance between {\it local independence} and {\it global dependence}
makes such systems not so easy to build (see however Example \ref{ex:smallN} and
Remark \ref{diffint} below for a simple case). This is the main point of our
work.

\subsection{Intricacies}
Throughout this paper, a {\bf system} is a finite collection $(X_i)_{i\in I}$ of random variables, each $X_i$, $i\in I$,
taking value in the same finite set, say $\{0,\dots,d-1\}$ with $d\geq 2$ given.
Without loss of generality, we suppose that $I$ is a subset of the positive
integers or simply $\{1,\dots,N\}$. In this case it is convenient to write
$N$ for $I$.

We let $\cX(d,I)$ be the set of such systems and $\cM(d,I)$
the set of the corresponding laws, that is, all probability measures on
$\{0,\dots,d-1\}^I$ for any finite subset $I$.
We often identify it with $\cM(d,N):=\cM(d,\{1,\dots,N\})$ for $N=|I|$.
If $X$ is such a system with law $\mu$, we denote its entropy by $\H(X)=\H(\mu)$.
Of course, entropy is in fact a (deterministic) function of the
law $\mu$ of $X$ and not of the (random) values of $X$.

Intricacies are functionals over such systems (more precisely: over their laws)
formalizing and generalizing the neural complexity \eqref{est} of
Edelman-Sporns-Tononi \cite{Tononi94}:

\begin{definition}\label{def1}
A {\bf system of coefficients} is a family of numbers
%over some set of indices $\mathbb I$ (typically $\mathbb I=\N^*$) is a collection of numbers
\[
c:=(c_S^I:\, I\subset\subset\N^*, \, S\subset I)
\]
%(i.e., $I$ ranges over the finite subsets of $\mathbb I$)
satisfying, for all $I$ and all $S\subset I$:
 \begin{equation}\label{eq:cond-CIS}
     c_S^I\geq 0, \quad \sum_{S\subset I} c_S^I = 1, \quad \text{ and } c^I_{S^c}=c^I_S
 \end{equation}
where $S^c:=I\setminus S$. We denote the set of such systems by $\SoC(\N^*)$.

The corresponding {\bf mutual information functional} is $\I^c:\cX\to\mathbb R$ defined by:
 $$
    \I^c(X):=\sum_{S\subset I} c^I_{S} \MI\left(X_S,X_{S^c}\right).
 $$
By convention,
$\MI\left(X_\emptyset,X_I\right)=\MI\left(X_I,X_\emptyset\right)=0$.
If $X\in\cX(d,I)$ has law $\mu$, we denote $\I^c(X)=\I^c(\mu)$.
$\I^c$ is {\rm \bf non-null} if  some coefficient $c^I_S$ with $S\notin\{\emptyset, I\}$ is not zero.

An {\bf intricacy} is a mutual information functional satisfying:
 \begin{enumerate}
  \item {\bf exchangeability} (invariance by permutations): if $I,J\subset\subset\N^*$
  and $\phi:I\to J$ is a bijection, then $\I^c(X)=\I^c(Y)$ for any $X:=(X_i)_{i\in I}$, $Y:=(X_{\phi^{-1}(j)})_{j\in J}$;
  \item {\bf weak additivity}:  $\I^c(X,Y)=\I^c(X)+\I^c(Y)$ for any two \emph{independent} systems $(X_i)_{i\in I},(Y_j)_{j\in J}$.
 \end{enumerate}
%Let $\I(\mathbb I)\subset\SoC(\mathbb I)$ denote the systems of coefficients defining intricacies.

%{\rm \bf trivial} if $c^I_S=0$ for all $I\subset\subset\N^*$ and $\emptyset\subsetneq S\subsetneq I$.
\end{definition}
Clearly, by \eqref{est}, neural complexity is a
mutual information functional with
$c^I_{S}=\frac1{|I|+1} \frac1{\binom{|I|}{|S|}}$,
satisfying exchangeability. Weak additivity is less trivial
and will be deduced in Theorem \ref{main1} below.
We remark that the factor $(|I|+1)$ in the denominator
is not present in the original definition in \cite{Tononi94}
but is necessary for weak additivity and the normalization
\eqref{eq:cond-CIS} to hold.

%We remark that we could relax the condition that each
%$X_i$ takes values in $\{0,\dots,d-1\}$ to $\H(X_i)\leq
%\log d$. This is obvious for Theorems \ref{} and \ref{thm:main}.

%Let us also point out that, except for the more precise
%Theorem \ref{thm:exchsta}, our results would still hold
%after replacing the entropy by any sub-additive functional
%$\phi$ uniformly bounded on r.v. taking a prescribed number
%of distinct values.

\subsection{Main results}

Our first result is a characterization of systems of coefficients $c$
generating an intricacy, i.e. an exchangeable and weak additive mutual
information functional. These properties are equivalent to a probabilistic
representation of $c$.

We say that a probability measure $\lambda$ on $[0,1]$ is {\bf symmetric} if $\int_{[0,1]} f(x)\, \lambda(dx)=
\int_{[0,1]} f(1-x) \lambda(dx)$ for all measurable and bounded functions $f$.
\begin{theorem}\label{main1}
Let $c\in\SoC(\N^*)$ be a system of coefficients and $\I^c$ the associated mutual information functional.
\begin{enumerate}
\item $\I^c$ is
an intricacy, i.e. exchangeable and weakly additive, if and only if there exists a symmetric probability measure
$\lambda_c$ on $[0,1]$ %w.r.t. $1/2$\footnote{$\lambda_c$ is
%symmetric w.r.t. $1/2$ iff $\lambda_c(f)=\lambda_c(f(1-\cdot))$ for all Borel bounded functions $f:[0,1]\to\mathbb R$.}
such that
\begin{equation}\label{exploo0}
c^I_S = \int_{[0,1]} x^{|S|}(1-x)^{|I|-|S|}\, \lambda_c(dx), \qquad
\forall \ S\subseteq I.
\end{equation}
In this case, if $\{W_c,Y_i, i\in\N^*\}$ is an independent family such that $W_c$ has law $\lambda_ c$
and $Y_i$ is uniform on $[0,1]$, then
\[
c^I_S = %\E\left(W_c^{|S|}(1-W_c)^{|I|-|S|}\right)\, \lambda_c(dx) =
\bbP(\z\cap I=S), \qquad
\forall \, I\subset\subset\N^*, \ \forall \, S\subset I,
\]
where $\z$ is the random subset of $\N^*$
\[
\z := \{i\in\N^*: Y_i\geq W_c\}.
\]
\item $\lambda_c$ is uniquely determined by $\I^c$. Moreover $\I^c$ is non-null
iff $\lambda_c(]0,1[)>0$ and in this case $c^I_S>0$ for all coefficients with $S\subset I$,
$S\notin\{\emptyset,I\}$.
\item For the neural complexity \eqref{est}, we have
\[
\frac1{|I|+1} \frac1{\binom{|I|}{|S|}} =
\int_{[0,1]} x^{|S|}(1-x)^{|I|-|S|}\, dx, \qquad
\forall \ S\subseteq I,
\]
i.e., $\lambda_c$ in this case is the Lebesgue measure on $[0,1]$
and the neural complexity is indeed exchangeable and weakly additive, i.e.
an intricacy.
\end{enumerate}
\end{theorem}
We discuss other explicit examples in section 2 below.

\medskip
Our next result concerns the maximal value of intricacies.
As discussed above, this is a subtle issue since large intricacy values require
compromises. This can also be seen in that intricacies are differences between entropies,
see \eqref{eq:I-as-H} and therefore not concave.

The weak additivity of intricacies is the key to how they grow with
the size of the system. This property of neural complexity having
been brought to the fore, we obtain linear growth and convergence of
the growth speed quite easily. The same holds subject to an entropy
condition, independently of the softness of the constraint (measured
below by the speed at which $\delta_N$ converges to $0$).

Denote by $\I^c(d,N)$ and $\I^c(d,N,x)$, $x\in[0,1]$, the supremum
of $\I^c(X)$ over all $X\in\cX(d,N)$, respectively over all
$X\in\cX(d,N)$ such that $\H(x)=xN\log d$:
\begin{equation}\label{supac0}
\I^c(d,N) := \sup\{\I^c(\mu):\mu\in\cM(d,N)\},
\end{equation}
\begin{equation}\label{supac1}
\I^c(d,N,x) := \sup\{\I^c(\mu): \mu\in\cM(d,N),\, \H(\mu)=xN\log
d\}.
\end{equation}
Notice that if $x=0$ or $x=1$, then $\I^c(d,N,x)=0$, since this
corresponds to, respectively, deterministic or independent systems,
for which all mutual information functionals vanish.
\begin{theorem}\label{thm:main}
Let $\I^c$ be a non-null intricacy and let $d\geq 2$ be some integer.
\begin{enumerate}
\item The following limits exist for all $x\in[0,1]$
\begin{equation}\label{omer0}
   \I^c(d):=\lim_{n\to\infty} \frac{\I^c(d,N)}N,
   \qquad    \I^c(d,x):=\lim_{n\to\infty} \frac{\I^c(d,N,x)}N,
\end{equation}
and we have the bounds
\begin{equation}\label{omer00}
%\frac{\kappa_c}2\, \leq
%\frac{\I^c(d)}{\log d} \leq \frac{1}2, \qquad
\left[x\wedge(1-x)\right] \kappa_c\, \leq
\frac{\I^c(d,x)}{\log d} \leq
\frac{\I^c(d)}{\log d}\leq \frac{1}2,
\end{equation}
where
\begin{equation}\label{ppa}
\kappa_c:=2\int_{[0,1]} y(1-y)\, \lambda_c(dy)>0,
\end{equation}
and $\lambda_c$ is defined in Theorem \ref{main1}.
\item Let $(\delta_N)_{N\geq1}$ be any sequence of non-negative numbers
converging to zero and $x\in[0,1]$. Then
 $$
  \I^c(d,x)= \lim_{N\to\infty}
            \frac1{N} \, \sup \left\{ \I^c(X): X\in\cX(d,N), \
            \left|\frac{\H(X)}{N\log d} - x\right| \leq \delta_N
            \right\}.
 $$
\end{enumerate}
\end{theorem}

\begin{remark}
{\rm $ $

1. By considering a set of independent, identically distributed (i.i.d. for short)
random variables on $\{0,\ldots,d-1\}$, it is easy
to see that for any $0\leq h\leq N\log d$, there is $X\in\cX(d,N)$ such that
$\H(X)=h$ and $\I^c(X)=0$. Hence minimization of intricacies is a trivial
problem also under fixed entropy.

2. It follows that for any $(x,y)$, $0\leq x\leq 1$ such that $0\leq y<
\I^c(d,x)/\log d$,
for any $N$ large enough, there exists $X\in\cX(d,N)$ with $\H(X)=xN\log d$
and $\I^c(X)=yN\log d$. Observe, for instance, that $\I^c$ is continuous
on the contractile space $\cM(d,N)$.

3. In the above theorem, the assumption that each variable $X_i$ takes values
in a set of cardinality $d$ can be relaxed to $\H(X_i)\leq\log d$. It can be shown
that this does not change $\I^c(d)$ or $\I^c(d,x)$.
}
\end{remark}

Thus maximal intricacy grows linearly in the size of the system.
What happens if we restrict to smaller classes of systems, enjoying
particular symmetries?
%We consider two interesting classes: stationary sequences
%and (finite) exchangeable families.
Since intricacies are exchangeable, their value does not change
if we permute the variables of a system. Therefore it is particularly
natural to consider (finite) exchangeable families.

We denote by $\EX(d,N)$ the set of random variables $X\in\cX(d,N)$ which
are  exchangeable, i.e., for all permutations $\sigma$ of $\{1,\ldots,N\}$,
$X:=(X_1,\dots,X_N)$ and $X_\sigma:=(X_{\sigma(1)},\dots,X_{\sigma(N)})$ have
the same law.
%We also consider stationary sequences $S\in\cS(d,\N^*)$, namely
%sequences $(S_i)_{i\geq 1}$ such that
%$S_i\in\{0,\ldots,d-1\}$ for all $i$ and, for any $N\geq 1$,
%$(S_i)_{i\geq 1}$ and $(S_{i+N})_{i\geq 1}$ have the same law.

\begin{theorem}\label{thm:exchsta}
Let $\I^c$ be an intricacy.
\begin{enumerate}
\item Exchangeable systems have small intricacies. More precisely
 $$
   \sup_{X\in\EX(d,N)} \I^c(X) =o(N^{2/3+\eps}), \qquad N\to+\infty,
 $$
for any $\eps>0$. In particular
 $$
    \lim_{N\to\infty} \frac1N \max_{X\in\EX(d,N)} \I^c(X) = 0.
 $$
\item For $N$ large enough and fixed $d$, maximizers of
$\cX(d,N)\ni X\mapsto \I^c(X)$ are neither unique nor exchangeable.
%
%\item For all $\gep>0$ there exists a stationary sequence $S\in\cS(d,\N^*)$ such that
%\[
%\limsup_{N\to+\infty} \frac{\I^c(S_1,\ldots,S_N)}N\geq \I^c(d)-\gep,
%\]
%where $\I^c(d)$ is defined in \eqref{omer0}.
%
\end{enumerate}
\end{theorem}
By the first assertion, exchangeability of the intricacies is not inherited by
their maximizers. Indeed, exchangeable systems are very far from maximizing, since the maximum
of $\I^c$ over $\EX(d,N)$ is $o(N^p)$ for any $p>2/3$ whereas the maximum
of $\I^c$ over $\cX(d,N)$ is proportional to $N$.
This "spontaneous symmetry breaking"  again suggests the complexity
of the maximizers. We remark that numerical estimates suggest that the intricacy of any $X\in \EX(d,N)$
is in fact bounded by $\operatorname{const}\log N$.

The second assertion of Theorem \ref{thm:exchsta} follows from the first one:
%A consequence is that exact maximizers at fixed system size are not unique:
for $N$ sufficiently large, the maximal intricacy is not
attained at an exchangeable law; therefore, by permuting a system with
maximal intricacy we obtain different laws, all with the same maximal
intricacy.

We finally turn to a property of exact maximizers, namely that their support is
concentrated on a small subset of all possible configuration:

\begin{theorem}\label{prop:support}
Let $\I^c$ be a non-null intricacy. let $d\geq2$. For $N$ a large enough integer,
the following holds. For any $X$ maximizing $\I^c$ over $\cX(d,N)$,
law $\mu$ of $X$ has small support, i.e.
 $$
     \#\{\omega\in\Lambda_{d,N}:\mu(\{\omega\})=0\} \geq \operatorname{const} d^N
 $$
for some $\operatorname{const}>0$.
\end{theorem}

\subsection{Further questions}
As noted above, the exact computation of the functions $\I^c(d)$ and $\I^c(d,x)$ from
Theorem \ref{thm:main} in terms of their probabilistic representation from Theorem \ref{main1} will be the subject of \cite{BZ2} where
we shall study systems with intricacy close to the maximum.

\smallbreak

Second, to apply intricacy one needs to compute it for systems of interests. It might be possible
to compute  it exactly for some simple physical systems, like the Ising
model. A more ambitious goal would be to consider more complex models, like spin glasses,
to analyze the possible relation between intricacy and frustration \cite{Tala}.

A more general approach would be to get rigorous estimates from numerical ones
(see \cite{Tononi94} for some rough computations). A naive approach results in an exponential
complexity and thus begs the question of more efficient algorithms,
perhaps probabilistic ones. A related question is the design
of statistical estimators for intricacies. These estimators should
be able to decide many-variables correlations, which might
require {\it a priori} assumptions on the systems.

\smallbreak

Third, one would to understand the intricacy from a dynamical point of view: which physically reasonable processes (say with dynamics defined in terms
of local rules) can lead to high intricacy systems and at what speeds?

\smallbreak

Fourthly, one could consider the  natural generalization of intricacies, already proposed in \cite{Tononi94} but not
explored further, is given in terms of general partitions $\pi$ of $I$:
if $\pi=\{S_1,...,S_k\}$ with $\cup_iS_i=I$ and $S_i\cap S_j=\emptyset$
for $i\ne j$, then we can set
\begin{equation}\label{kin1}
\MI(X_\pi):=\H(X_{S_1})+\cdots+\H(X_{S_k})-\H(X), \qquad X\in\cX(d,I),
\end{equation}
and for some non-negative coefficients $(c_\pi)_\pi$
\begin{equation}\label{kin2}
{\mathcal J}^c(X):= \sum_\pi c_\pi \MI(X_\pi).
\end{equation}
Most results of this paper extend to the case where
the coefficients $(c_\pi)_\pi$ have a probabilistic representation in terms of the so-called
Kingman paintbox construction \cite[\S 2.3]{bertoin}, see Remark \ref{kingman} below.

One might also be interested to extend the definition of intricacy to
infinite (e.g., stationary) processes, continuous or structured systems, e.g., taking
into account a connectivity or dependence graph (such constraints have been
considered in numerical experiments performed by several authors
\cite{Barnett09,DeLuca04,Sporns00}).

\smallbreak

Finally, our work leaves out the properties of exact maximizers for a given
size. As of now, we have no description of them except in very special
cases (see Examples \ref{ex:smallN2} and \ref{ex:smallN} below) and we do not know how many
there are, or even if they are always in finite number.
We do not have reasonably efficient ways to determine the maximizers
which we expect to lack a simple description in light of the
lack of symmetry established in Theorem \ref{thm:exchsta}.

\subsection{Organization of the paper}

In Sec. 2, we discuss the definition of intricacies, giving some basic properties
and examples. Sec. 3 proves Theorem \ref{main1}, translating the weak additivity of an intricacy
into a property of its coefficients. As a by-product, we obtain a probabilistic
representation of all intricacies. We check that neural
complexity corresponds to the uniform law on $[0,1]$.
In Sec. 4 we prove Theorem \ref{thm:main} by showing the
existence of the limits $\I^c(d)$, $\I^c(d,x)$. Finally, in Sec. 5
we prove Theorem \ref{thm:exchsta} and, in Sec. 6, Theorem \ref{prop:support}.
An Appendix recalls some basic facts from information theory for the convenience
of the reader and to fix notations.

%%%%%%%%%%%%%%%%%%%%%%%%%%%%%%%%%%%%%%%%%%%%%%%%%%%%%%%%%%%%%%%%%%%%%%%%%%%%%%%%%%
%%%%%%%%%%%%%%%%%%%%%%%%%%%%%%%%%%%%%%%%%%%%%%%%%%%%%%%%%%%%%%%%%%%%%%%%%%%%%%%%%%
\section{Intricacies}\label{sec:intricacy}
%%%%%%%%%%%%%%%%%%%%%%%%%%%%%%%%%%%%%%%%%%%%%%%%%%%%%%%%%%%%%%%%%%%%%%%%%%%%%%%%%%
%%%%%%%%%%%%%%%%%%%%%%%%%%%%%%%%%%%%%%%%%%%%%%%%%%%%%%%%%%%%%%%%%%%%%%%%%%%%%%%%%%

\subsection{Definition}
We begin by a discussion of the definition \ref{def1} above of intricacies.
%For all intricacies, and more generally exchangeable mutual information
%functionals, $c\in\SoC(\mathbb I)$ readily extends to any set with
%cardinality at most that of $\mathbb I$.
As $\MI(X_S,X_{S^c})=\MI(X_{S^c},X_S)$, the symmetry condition
$c^I_{S^c}=c^I_S$ can always be satisfied by replacing $c^I_S$ with
$\frac12(c^I_S+c^I_{S^c})$ without changing the functional. Also
$\sum_{S\subset I} c^I_S=1$ is simply an irrelevant normalization when
studying systems with a given index set $I$.

The following mutual information functionals will be proved to be
intricacies in section \ref{additivity}.

\begin{definition}
The intricacy $\I$ of {\bf Edelman-Sporns-Tononi}  is defined by its
coefficients:
\begin{equation}\label{ET}
    c^I_S = \frac{1}{|I|+1}\frac{1}{\binom{|I|}{|S|}}.
\end{equation}
For $0<p<1$, the {\bf $p$-symmetric intricacy} $\I^p(X)$ is:
 $$
     c^I_S = \frac12\left(p^{|S|}(1-p)^{|I\backslash S|}+(1-p)^{|S|}
\, p^{|I\backslash S|}\right).
 $$
For $p=1/2$, this is the {\bf uniform intricacy} $\I^U(X)$ with:
 $$
     c^I_{S}=2^{-|I|}.
 $$
\end{definition}
It is not obvious that the three above mutual information
functionals are weakly additive, but this will follow easily
from Lemma \ref{1012} below. Proposition \ref{1916} below describes all intricacies.

\begin{rem}
{\rm The coefficients of the Edelman-Sporns-Tononi
intricacy $\I$ ensure that subsystems of all sizes
contribute significantly to the intricacy. This is in sharp contrast to the
$p$-symmetric coefficients for which subsystems of size far from
$pN$ or $(1-p)N$ give a vanishing contribution when $N$ gets large.}
\end{rem}

\begin{rem}
{\rm The global $1/(|I|+1)$ factor in $\I$ is not present in
\cite{Tononi94}, which did not compare systems of different sizes.
However it is required for weak additivity.
}
\end{rem}

\subsection{Basic Properties}
We prove some general and easy properties of intricacies. Recall
that $\cX(d,N)$ is the set of $\Lambda_{d,N}$-valued random
variables, where $\Lambda_{d,N}=\{0,\dots,d-1\}^N$. We identify
it with the standard simplex in $\mathbb R^{d^N}$ in the
obvious way.
%It is a convex subset of $\mathbb R^{d^N}$.

%
\begin{lemma}\label{assumed}
Let $\I^c$ be a mutual information functional. For each $d\geq2$ and
$N\geq1$, $\I^c:\cM(d,N)\to \mathbf R$ is continuous. In particular,
the suprema $\I^c(d,N)$ and $\I^c(d,N,x)$, introduced in
\eqref{supac0} and \eqref{supac1}, are achieved.

If $\I^c$ is a non-null intricacy, then it is neither convex nor concave.
\end{lemma}
\begin{proof}
Continuity is obvious and existence of the maximum follows from the
compactness of the finite-dimensional simplex $\cM(d,N)$.
To disprove convexity and concavity of non-null intricacies,
we use the following examples.
Pick $I$ with at least two elements, say $1$ and
$2$. Observe that $K:=c_{\{1\}}^I+c_{\{2\}}^I$ is positive by the
non-degeneracy of $\I^c$ (see Lemma \ref{alors} below). Fix $d\geq2$.

First, for $i=0,1$, let $\mu_i$ over $\{0,\dots,d-1\}^I$
be defined by $\mu_i(i,i,0,\dots,0)=1$. We have:
 $$
    \I^c\left(\frac{\mu_0+\mu_1}2\right) \geq K \cdot \log d
     > \frac{\I^c(\mu_0)+\I^c(\mu_1)}2 = 0.
 $$
Second, let $\nu_0$ be defined by $\nu_0(0,0,0,\dots,0)=\nu_0(1,1,0,\dots,0)=1/2$ and
$\nu_1$ by $\nu_1(0,1,0,\dots,0)=\nu_1(1,0,0,\dots,0)=1/2$. We have:
 $$
   \I^c\left(\frac{\nu_0+\nu_1}2\right) = 0
     <  K \cdot \log d \leq \frac{\I^c(\nu_0)+\I^c(\nu_1)}2.
 $$
\end{proof}

\noindent
The following expression of an intricacy as a non-convex combination of the entropy
of subsystems is crucial to its understanding.
\begin{lemma}
For any intricacy $\I^c$ and $X\in\cX(d,N)$
 \begin{equation}\label{eq:I-as-H}
     \I^c(X) %=  \sum_S c_S^I \MI(X_S,X_{S^c}) %= \sum_S c_S \H(X_S)+\H(X_{S^c}) - \H(X)
         = 2\left(\sum_{S\subset I} c_S^I \H(X_S)\right) - \H(X).
 \end{equation}
\end{lemma}

\begin{proof}
The result readily follows from: $\MI(X,Y)=\H(X)+\H(Y)-\H(X,Y)$, $c^I_S=c^I_{S^c}$, and $\sum_S c^I_S=1$.
\end{proof}

\noindent
We introduce the notation
\[
\MI(S):=
\MI(X_S,X_{I\setminus S})
\]
which will be used only when the understood dependence on $X$ and $I$ is clear.

\begin{lemma}\label{bert}
For any intricacy $\I^c$ and any system $X\in\cX(d,N)$
 $$
    0\leq \I^c(X) \leq \frac{N}{2}\log d.
 $$
\end{lemma}

\begin{proof}
The inequalities follow from basic properties of the mutual information (see the Appendix):
 $$
    0\leq \MI(S)\leq \min\{\H(X_S),\H(X_{S^c})\}
      \leq \min\{|S|,N-|S|\} \log d \leq \frac{N}2  \log d.
 $$
\end{proof}

\subsection{Simple examples}\label{tse}
We give some examples of finite systems and compute their intricacies both
for illustrative purposes and for their use in some proofs below.

Let $X_i$ take values in $\{0,\ldots,d-1\}$ for all $i\in I$,
a finite subset of $\N^*$. The first two examples show that
total order and total disorder make the intricacy vanish.

\begin{example}[Total disorder]\label{gb}
{\rm If the variables $X_i$ are independent then each mutual information
is zero and therefore:
 $
    \I^c(X)=0. \qquad \qquad \square
 $
}
\end{example}
\begin{example}[Total order]\label{gb2}
{\rm If each $X_i$ is a.s. equal to a constant $c_i$ in $\{0,\ldots,d-1\}$,
then, for any $S\ne\emptyset$, $\H(X_S)=0$. Hence,
 $
   \I^c(X) = 0. \qquad \qquad \square
 $
}
\end{example}
For $N=2,3$, each mutual information can be maximized separately: there is no frustration
and it is easy to determine the maximizers of non-null intricacies.
\begin{example}[Case $N=2$]\label{ex:smallN2}
{\rm
Let first $N=2$ and $\I^c$ be a non-null intricacy. Then by Theorem \ref{main1}
$c^I_S=c^{|I|}_{|S|}$ and therefore
\[
\I^c(X) = \left( c^{\{1,2\}}_{\{1\}}+ c^{\{1,2\}}_{\{2\}} \right) \MI(X_1,X_2)
=2c^2_1 \, \MI(X_1,X_2), \qquad X\in\cX(d,2),
\]
and moreover $c^2_1>0$.
Therefore the maximizers of $\I^c$ over $\cX(d,2)$ are the maximizers of $X\mapsto\MI(X_1,X_2)$.
By the discussion in subsection \ref{mi} of the appendix, we have that $\MI(X_1,X_2)\leq\min\{\H(X_1),\H(X_2)\}$. Now,
$\MI(X,Y)=\H(X_1)=\H(X_2)$ iff each variable is a function of
the other.
%In order to maximize $\H(X_1)$ we consider $X_1$ uniform on $\{0,\dots,d-1\}$.

Therefore, the maximizers are exactly the following systems $X=(X_1,X_2)$.
$X_1$ is a uniform r.v. over $\{0,\dots,d-1\}$
and the other is a deterministic function of the first.
$X_2=\sigma(X_1)$ for a given permutation $\sigma$ of $\{0,\dots,d-1\}$.
In the case of the neural complexity, $\max_{X\in\cX(d,2)}\I(X)=(\log d)/3$.
\qquad \qquad $\square$
}
\end{example}
\begin{example}[Case $N=3$]\label{ex:smallN}
{\rm
Let $N=3$ and $I:=\{1,2,3\}$. By Theorem \ref{main1}, $c^I_S=c^{|I|}_{|S|}$, $c^3_1=c^3_2$
 and therefore
\[
\I^c(X) = 2c^3_1\left(\MI(X_1,X_{\{2,3\}})+\MI(X_2,X_{\{1,3\}})+\MI(X_3,X_{\{1,2\}})\right),
\]
and moreover $c^3_1>0$.
Here we simultaneously maximize each of these mutual informations.
The optimal choice is a system $(X_1,X_2,X_3)$ where every pair $(X_i,X_j)$, $i\ne j$,
is uniform over $\{0,\dots,d-1\}^2$,
and the third variable is a function of $(X_i,X_j)$. This is realized iff
$(X_1,X_2)$ is uniform over $\{0,\dots,d-1\}^2$
and $X_3=\phi(X_1,X_2)$, where $\phi$ is a (deterministic) map such that,
for any $i\in\{0,\dots,d-1\}$, $\phi(i,\cdot)$ and $\phi(\cdot,i)$ are permutations
of $\{0,\dots,d-1\}$. For instance: $\phi(x_1,x_2)=x_1+x_2\mod d$.
In the case of the neural complexity, $\max_{X\in\cX(d,3)}\I(X)=(\log d)/2$.
\qquad \qquad $\square$
}
\end{example}
The maximizers of  examples \ref{ex:smallN2} and \ref{ex:smallN} are very
special. For instance, they are exchangeable, contrarily to the case of
large $N$ according to Theorem \ref{thm:exchsta}. For $N=4$ and beyond
it is no longer possible to separately maximize each mutual information
and we do not have an explicit description of the maximizers. We shall however
see that, as in the above examples, maximizers have small support, see
Proposition \ref{prop:support}.
\begin{rem}\label{diffint}
{\rm
Example \ref{ex:smallN} has an interesting interpretation: for $N=3$,
a system with large intricacy shows in a simple way a combination of
{\it differentiation} and {\it integration}, as it is expected in the
biological literature, see the Introduction. Indeed,
any subsystem of two variables is independent (differentiation), while
the whole system is correlated (integration).
}
\end{rem}

Another interesting case is that of a large system where one variable is free and
all others follow it deterministically.
\begin{example}[A totally synchronized system]\label{gb3}
{\rm
Let $X_1$ be a uniform $\{0,\ldots,d-1\}$-valued random variable.
We define now $(X_2,\ldots,X_N):=\phi(X_1)$, where
$\phi$ is any deterministic map from $\{0,\ldots,d-1\}$ to $\{0,\ldots,d-1\}^{N-1}$. Then,
for any $S\ne\emptyset$, $\H(X_S)=\log d$ and, if additionally $S^c\ne\emptyset$,
$\H(X_S|X_{S^c})=0$ so that each mutual information $\MI(X_S,X_{S^c})$ is $\log d$ if
$S\notin\{\emptyset,I\}$. Hence,
 $$
   \I^c(X) = \sum_{S\subset I\backslash\{\emptyset, I\}} c_S^I \cdot \log d
=\left(1-c_\emptyset^I-c_I^I\right) \log d. \qquad \qquad \square
 $$
}
\end{example}
In the next example we build for every $x\in\,]0,1[$ a system $X\in\cX(d,2)$ with
entropy $\H(X)=x\log d^2$ and positive intricacy.
\begin{example}[A system with positive intricacy and arbitrary entropy]\label{gb3.5}
{\rm Let first $x\in\,]0,1/2]$. Let $X_1$ be $\{0,\ldots,d-1\}$-valued with $\H(X_1)=2x\log d$.
Such a variable exists because entropy is continuous over the
connected simplex of probability measures on $\{0,\ldots,d-1\}$ and attains the values $0$
over a Dirac mass and $\log d$ over the uniform distribution. We define now $X_2:=X_1$ and
$X:=(X_1,X_2)\in\cX(d,2)$. Therefore $\H(X)=2x\log d=x\log d^2$,
$\MI(X_1,X_2)=\H(X_1)=2x\log d$ and, arguing as in Lemma \ref{ex:smallN2}
\[
\I^c(X)=2c^2_1 \, \MI(X_1,X_2)=4x\,c^2_1 \, \log d>0.
\]

Let now $x\in\,]1/2,1[$.
Let $(Y_1,Y_2,B)$ be an independent triple such that $Y_i$ is uniform over $\{0,\ldots,d-1\}$
and $B$ is Bernoulli with parameter $p\in[0,1]$ and set
\[
X_1:=Y_1, \qquad X_2:=\un{(B=0)}\,Y_1+\un{(B=1)}\,Y_2, \qquad  X:=(X_1,X_2).
\]
Then both $X_1$ and $X_2$ are uniform on $\{0,\ldots,d-1\}$. On the other hand, it is easy to
see that $\H(X)$, as a function of $p\in[0,1]$, interpolates continuously between $\log d$ and $2\log d$.
Thus, for every $x\in\,]1/2,1[$ there is a $p\in[0,1]$ such that $\H(X)=x\log d^2$. In this case
$\MI(X_1,X_2)=2(1-x)\log d$ and we obtain
\[
\I^c(X)=2c^2_1 \, \MI(X_1,X_2)=4(1-x)\,c^2_1 \, \log d>0. \qquad \square
\]
}
\end{example}

Intricacy can indeed reach over $\cX(d,N)$ the order $N$ of Lemma
\ref{bert}, as the next example shows.

\begin{example}[Systems with uniform intricacy proportional to $N$]\label{gb5}
{\rm Let us fix $d\geq2$. For $N\geq 2$, we are going to build a
system $(X_i)_{i\in I}$, $I=\{1,\dots,N\}$, over the alphabet
$\{0,\dots,d^2-1\}$ for which $\I^U(X)/N$ converges to $(\log d^2)/4$; later, in
Example \ref{gb5a}, we shall generalize this to an arbitrary
intricacy.

Let $Y_1,\dots,Y_{N}$ be i.i.d. uniform $\{0,\dots,d-1\}$-valued
random variables and define $X_i:=Y_i+dY_{i+1}$ for $i=1,\dots,N-1$,
$X_N:=Y_N$. Note that $X\in\cX(d^2,N)$ and $\H(X)=N\log d=(N/2)\log
d^2$. For $S\subset I$, set
\[
\begin{split}
\Delta_S & :=\{k=1,\dots,N-1:\un{S}(k)\ne \un{S}(k+1)\},
\\ U_S & :=\{k=1,\dots,N-1:\un{S}(k)=1\ne \un{S}(k+1)\}.
\end{split}
\]
Observe that $\H(X_S)=(|S|+|U_S|)\log d$. Indeed, this is given by
$\log d$ times the minimal number of $Y_i$ needed to define $X_S$;
every $k\in S$ counts for one if $k\in S\setminus U_S$, for two if
$k\in U_S$; therefore we find $|S|-|U_S|+2 |U_S|=|S|+|U_S|$.
Moreover, $|U_S|+|U_{S^c}|=|\Delta_S|$. Therefore
%Since $\MI(S)=\MI(S^c)$, we can suppose e.g. that $N\in S$ and obtain
\[
\begin{split}
\MI(S) & =( |U_S|+|S| + |U_{S^c}|+|S^c|-N)\log d = |\Delta_S| \,
\log d.
\end{split}
\]
Moreover we have a bijection:
 $$
  S\in \{0,1\}^{\{1,\dots,N\}}\mapsto (\un{S}(1),\Delta_S)\in\{0,1\}\times \{0,1\}^{\{1,\dots,N-1\}}.
 $$
Hence:
\[
\begin{split}
    \frac{\I^U(X)}{\log d} & =2^{-N}\sum_{S\subset I} |\Delta_S|
      = 2^{-N} \times 2\sum_{\Delta\subset\{1,\dots,N-1\}} |\Delta|
= 2^{-N+1} \sum_{k=0}^{N-1} \binom{N-1}{k} k
 \\ &       = 2^{-N+1} (N-1) 2^{N-2}       = \frac{N-1}{2}.
\end{split}
\]
Therefore for this $X\in\cX(d^2,N)$:
\[
\qquad\qquad\qquad \I^U(X) = \frac{N-1}{4} \,\log (d^2). \qquad \qquad \qquad \qquad\square
\]
}
\end{example}

The following example will be useful to show that an intricacy $\I^c$
determines its coefficients $c\in\SoC(\N^*)$ in Lemma \ref{lem:uniq-coef}
below. %See also Example \ref{gb4a} for a further computation.
\begin{example}[A system with a synchronized sub-system]\label{gb4}
{\rm We consider a system of uniform variables,
with a subset of equal ones and the remainder independent. More precisely,
let $I\subset\subset\N^*$, $\emptyset\ne K\subset I$ and fix $i_0\in K$.
$(X_i)_{i\in I}\in\cX(d,I)$ is the system satisfying:
\begin{itemize}
\item[(i)] the family $X_{K^c\cup\{i_0\}}$ is uniform on $\{0,\ldots,d-1\}^{K^c\cup\{i_0\}}$;
\item[(ii)] $X_i=X_{i_0}$ for all $i\in K$.
\end{itemize}
It follows that
\[
\H(X_S) = \left( |S\setminus K|+\un{(S\cap K\ne \emptyset)}\right) \log d
\]
and therefore
\[
\MI(S) = \left( \un{(S\cap K\ne \emptyset)}+\un{(S^c\cap K\ne \emptyset)}-1\right) \log d,
\]
i.e. $\MI(S)=0$ unless $S$ and
$S^c$ both intersect $K$ and then $\MI(S)=\log d$. Thus
 $$
    \I^c(X) = \log d \sum_{S\subset I} c^I_S \, \un{(\emptyset\ne S\cap K\ne K)},  \qquad
    \H(X) = (|K^c|+1)\log d. \qquad \square
 $$
 }
\end{example}

%%%%%%%%%%%%%%%%%%%%%%%%%%%%%%%%%%%%%%%%%%%%%%%%%%%%%%%%%%%%%%%%%%%%%%%%%%%%%%%%%%
%%%%%%%%%%%%%%%%%%%%%%%%%%%%%%%%%%%%%%%%%%%%%%%%%%%%%%%%%%%%%%%%%%%%%%%%%%%%%%%%%%
\section{Weak additivity, projectivity and representation}\label{additivity}
%%%%%%%%%%%%%%%%%%%%%%%%%%%%%%%%%%%%%%%%%%%%%%%%%%%%%%%%%%%%%%%%%%%%%%%%%%%%%%%%%%
%%%%%%%%%%%%%%%%%%%%%%%%%%%%%%%%%%%%%%%%%%%%%%%%%%%%%%%%%%%%%%%%%%%%%%%%%%%%%%%%%%

In this section we prove Theorem \ref{main1}, by
studying the additivity of mutual information
functionals and characterizing it in terms of the coefficients.
We establish a probabilistic representation of all intricacies
and check that the neural complexity is indeed an intricacy.
We conclude this section by some useful consequences of this representation.

Throughout this section, $X=(X_i)_{i\in I}$ and $Y=(Y_i)_{i\in J}$, will be two systems
defined on the same probability space and we shall consider the joint family
$(X,Y)=\{X_i,Y_j:i\in I,j\in J\}$.
$(X,Y)$ is again a system and its index set is the disjoint union $I\sqcup J$ of $I$ and $J$.

\subsection{Projectivity and Additivity}\label{sec:proof-lem-special}

\noindent
We show that weak additivity and exchangeability can be read off the coefficients and that
non-null intricacies are neither sub-additive nor super-additive.
\begin{prop}\label{prop:product}
Let $\I^c$ be a mutual information functional. Then
 \begin{enumerate}
  \item $\I^c$ is exchangeable if and only if $c^I_S$ depends only on
  $|I|$ and $|S|$
 \item $\I^c$ is weakly additive if and only if the coefficients are {\bf projective}, i.e., satisfy
   \begin{equation}\label{mainc}
   \forall \, I\subset\subset\N^*, \ \forall \, J\subset\subset\N^*\setminus I, \
   \forall \, S\subset I, \qquad
   c^I_S = \sum_{T\subset J} c^{I\sqcup J}_{S\sqcup T}.
   \end{equation}
 \item Let $\I^c$ be an intricacy. Then,
for non-necessarily independent systems $X,Y$, we have:
$\I^c(X,Y)\geq\max\{\I^c(X),\I^c(Y)\}$ and the approximate
additivity:
  $$
    |\I^c(X)+\I^c(Y)-\I^c(X,Y)| \leq \MI(X,Y);
  $$
  \item $\I^c$ can fail to be super-additive or sub-additive.
 \end{enumerate}
\end{prop}

To prove this proposition we shall need the following fact:

\begin{lemma}\label{lem:uniq-coef}
Let $d\geq2$ and $I$ be a finite set. The data $\I^c(X)$ for $X\in\cX(d,J)$
for all $J\subset\subset I$ determine $c\in\SoC(I)$.
\end{lemma}

\begin{proof}
%The map is onto by definition. Let $c\in\SoC(\N^*)$ so
%that $\I^c$ is some mutual information fonctional.
Using $c^I_{S^c}=c^I_S$, we restrict ourselves to coefficients
with $|S|\leq|S^c|$, i.e., $|S|\leq|I|/2$.
Let us first consider a system $(X_i)_{i\in I}\in\cX(d,I)$ where all
variables are equal:
$X_i=X_j$ for all $i,j\in I$ and $X_i$ is uniform on $\{0,\ldots,d-1\}$. Then
$\MI(S):=\MI(X_S,X_{S^c})=0$ for $S=\emptyset$ or $S=I$, otherwise $\MI(S)=\log d$.
Hence, using the normalization $1=\sum_S c^I_S$:
 $$
    1-\frac{\I^c(X)}{\log d}=
      \sum_S c^I_S - \sum_{\emptyset\subsetneq S\subsetneq I} c^I_S
      = c^I_\emptyset+c^I_I.
 $$
In particular, $c^I_\emptyset=c^I_I=(1-\I^c(X)/\log d)/2$.

For each $K\subset I$, let $X^K$ be the system as in Example \ref{gb4}. Fix $i_0\in K$. Recall that $\MI(S):=\MI(X_S,X_{S^c})$ is $0$
if $S\supset K$ or $S^c\supset K$,  and is $\log d$ otherwise.
Assume by induction that, for $1\leq s\leq|I|/2$, $c^I_S$ is determined for $|S|<s$
(a trivial assertion for $s=1$).
Picking $K\subset I$ with $|K|=|I|-s\geq|I|/2\geq|K^c|=s$, we get:
 \begin{itemize}
  \item if $|S|<s$, we say nothing of $\MI(S)$ but will use the
  inductive assumption;
  \item if $S=K$ or $S=K^c$, then $\MI(S)=0$;
  \item if $s\leq |S|\leq |K|$,  $S\supset K$ implies $S=K$,
  $S\subset K^c$ implies $S=K^c$ since $s=|K^c|$. In all other
  cases: $\MI(S)=\log d$.
 \end{itemize}
Therefore,
\[
 \begin{split}
    \frac{\I^c(X^K)}{\log d} & = 2\sum_{S\subset I} c^I_S\frac{\MI(S)}{\log d}
      - \frac{\H(X^K)}{\log d}
 \\ & = 4\sum_{|S|<|I|/2} c^I_S\frac{\MI(S)}{\log d}
     + 2\sum_{|S|=|I|/2} c^I_S\frac{\MI(S)}{\log d} - \frac{\H(X^K)}{\log d}
\\ & =
4\sum_{|S|<s} c^I_S\frac{\MI(S)}{\log d}
       + 4\sum_{s\leq |S|<|I|/2} c^I_S + 2\sum_{|S|=|I|/2} c^I_S
       \; - \; 2(c^I_K + c^I_{K^c})
       - \frac{\H(X^K)}{\log d}
       \end{split}
 \]
(the sum over $|S|=|I|/2$ is non-zero only if $|I|$ is even). Using $\sum_S c^I_S=1$ and
$c^I_S=c^I_{S^c}$,
we get:
 $$
  \frac{\I^c(X^K)}{\log d}+\frac{\H(X)}{\log d}-2
= 2\sum_{S\subset I} c^I_S\left(\frac{\MI(S)}{\log d}-1\right)
= 4\sum_{|S|<s} c^I_S\left(\frac{\MI(S)}{\log d}-1\right)
        - 4c^I_{K^c} .
 $$
It follows that $c^I_K=c^I_{K^c}$ is determined for any $K$ with $|K|=s$. This completes the induction step and the proof of the lemma.
\end{proof}

\begin{proof}[Proof of Proposition \ref{prop:product}]
The characterization
of exchangeability is a direct consequence of Lemma \ref{lem:uniq-coef}.

Let us prove the second point.
We first check that weak additivity implies projectivity. For any $X\in\cX(d,I)$ with $I\subset\subset\N^*$ and $J\subset\subset\N^*\setminus I$, we have:
 $$
    \I^c(X)=\I^c(X,Z) = \sum_{S\subset I}\sum_{T\subset J} c^{I\cup J}_{S\cup T} \MI(X_S,X_{S^c})
 $$
for $Z=(Z_j)_{j\in J}$ with each $Z_j$ a.s. constant and therefore independent of $X$.
Lemma \ref{lem:uniq-coef} then implies that \eqref{mainc} holds.
Moreover, \eqref{eq:MI-mono} yields the monotonicity claim of point (2).

For the approximate additivity, we consider \eqref{eq:MI-add} for any $S\subset I$,
$T\subset J$:
 $$
  \MI((X_S,Y_T),(X_{S^c},Y_{T^c})) = \MI(X_S,X_{S^c})+\MI(Y_{T},Y_{T^c})
    \pm \MI(X,Y)
 $$
where $\pm \MI(X,Y)$ denotes a number belonging to $[-\MI(X,Y),\MI(X,Y)]$.
The projectivity now gives:
\[
\begin{split}
\I^c(X,Y) & = \sum_{S\subseteq I, T\subseteq J} c^{I\sqcup J}_{S\sqcup T}
\, \MI(S\sqcup T ) \\
 & = \sum_{S\subseteq I, T\subseteq J} c^{I\sqcup J}_{S\sqcup T}(
\MI(X_S,X_{S^c})+\MI(Y_T,Y_{T^c}) \pm\MI(X,Y))
\\ & = \I^c(X) + \I^c(Y) \pm \MI(X,Y).
\end{split}
\]
This is the approximate additivity of point (2).
If $X$ and $Y$  are independent, then $\MI(X,Y)=0$, proving the weak additivity.

We finally give the counter-examples.
For sub-additivity, it is enough to assume the intricacy to be non-null and
to consider $X=Y$ a single random variable uniform on $\{1,2\}$ and compute:
 $$
   \I^c(X)=\I^c(Y)=0 \text{ whereas }\I^c(X,Y)=2c^2_1\log 2>0.
 $$
For super-additivity, we assume
%\annotation{ADDITIONAL ASSUMPTION???}
${c_\emptyset^{I}+c_{I}^{I}}<\frac12+\frac{c_\emptyset^{I\sqcup I}+c_{I\sqcup I}^{I\sqcup I}}2$ and take $X=Y$ a collection of $N=|I|$ copies of the same variable
uniform over $\{0,1\}$.
Then $\MI(S)=\log 2$ except if $S\in\{\emptyset,I\}$, in
which case $\MI(S)=0$. By example \ref{gb3}
\[
\frac{\I^c(X,Y)}{\log 2} = 1-c_\emptyset^{I\sqcup I}-c_{I\sqcup I}^{I\sqcup I}
<2\left(1-c_\emptyset^{I}-c_{I}^{I}\right) = \frac{\I^c(X)+\I^c(Y)}{\log 2}.
\]
\end{proof}
%

%%%%%%%%%%%%%%%%%%%%%%%%%%%%%%%%%%%%%%%%%%%%%%%%%%%%%%%%%%%%%%%%%%%%%%%%%%%
\subsection{Probabilistic representation of intricacies}
%%%%%%%%%%%%%%%%%%%%%%%%%%%%%%%%%%%%%%%%%%%%%%%%%%%%%%%%%%%%%%%%%%%%%%%%%%%

In this section, we give the probabilistic representation for intricacies.
This will provide us with a way to estimate the maximal value of
intricacy for large systems in \cite{BZ2}. For notational convenience, we consider intricacies over
the positive integers $\N^*$.

We say that a random variable $W$ over $[0,1]$ is {\bf symmetric} if
$W$ and $1-W$ have the same law. A measure on $[0,1]$ is symmetric
if it is the law of a symmetric random variable.

\begin{prop}\label{1917}
Let $\I^c$ be a mutual information functional defined by some
system of coefficients $c\in\mathcal C(\N^*)$ over some
infinite index set, which we assume to be $\N^*$ for notational
convenience.
\begin{enumerate}
\item $\I^c$ is an intricacy,
i.e., it is exchangeable and weakly additive,
if and only if there exists a symmetric random variable $W_c$ over $[0,1]$
with law $\lambda_c$ such that for all $I\subset\subset\N^*$ and
$S\subset I$
%with law $\lambda_c$ such that %for all $S\subset I\subset\subset\N^*$
%
 \begin{equation}\label{eq:proj}
c^I_S = \E\left(W_c^{|S|}(1-W_c)^{|I|-|S|}\right)
= \int_{[0,1]} x^{|S|}(1-x)^{|I|-|S|}\, \lambda_c(dx).
 \end{equation}
\item Formula \eqref{eq:proj} is equivalent to
 \begin{equation}\label{eq:proj2}
c^I_S = %\E\left(W_c^{|S|}(1-W_c)^{|I|-|S|}\right)\, \lambda_c(dx) =
\bbP(\z\cap I=S), \qquad
\forall \, I\subset\subset\N^*, \ \forall \, S\subset I,
 \end{equation}
where $\z$ is the random subset of $\N^*$
\begin{equation}\label{explo}
\z := \{i\in\N^*: Y_i\geq W_c\},
\end{equation}
with $(Y_i)_{i\geq 1}$ an i.i.d. sequence of uniform random variables on $[0,1]$, independent of $W_c$.
\item If $\I^c$ is an intricacy, then the law
$\lambda_c$ of $W_c$ is uniquely determined by $\I^c$. Moreover
for all $X\in\cX(d,I)$ independent of $\z$
\[
\I^c(X) = \E(\MI(\z\cap I)), \qquad \MI(S):=\MI(X_S,X_{I\backslash S}).
\]
\end{enumerate}
\end{prop}
\begin{rem}\label{kingman}
{\rm The definition \eqref{explo} of the random set $\z$ is a particular
case of the so-called {\it Kingman paintbox} construction,  see \cite[\S 2.3]{bertoin}. In this setting,
it yields a random exchangeable partition of $\N^*$ into a subset $\z$ and its
complement, each with asymptotic density a.s. equal to $W_c$, respectively
$1-W_c$.
Therefore it is natural to expect a similar probabilistic representation
for coefficients $(c_\pi)_\pi$ of exchangeable and weakly additive generalized functionals defined in
\eqref{kin1} and \eqref{kin2}. 
}
\end{rem}
\noindent
After the proof of the proposition we give the measures
$\mu,\mu^U,\mu^p$ representing respectively
$\I$, $\I^U$ and $\I^p$.
We start with the following
\begin{lemma}\label{1916}
Let $\SoC(\N^*)$ be the set of systems of coefficients of intricacies.
Let $\mathcal P\mathcal S([0,1])$ be the set of symmetric probability measures $\lambda$ on $[0,1]$. Then, the map $\lambda\mapsto c$ defined from $\mathcal P\mathcal S([0,1])$ to $\SoC(\N^*)$
according to ($n:=|I|$, $k:=|S|$):
\begin{equation}\label{exploo}
c^I_S = c^n_k = \int_{[0,1]} x^k(1-x)^{n-k}\, \lambda(dx),
 \qquad   \forall S\subset I\subset\subset\N^*,
\end{equation}
is a bijection.
\end{lemma}

\begin{proof}[Proof of Lemma \ref{1916}]
We first show that for an exchangeable weakly additive $\I^c$,
there exists a probability measure $\lambda$ on $[0,1]$ such that
\begin{equation}\label{pages}
c^{n+k}_n = \int_{[0,1]} x^n(1-x)^k\, \lambda(dx), \qquad n\geq 1, \, k\geq 0
\end{equation}
i.e., the main claim of the Lemma, up to a convenient renumbering.
We need the following classical moment result, see e.g. \cite[VII.3]{moment}.
\begin{lemma}
Let $(a_n)_{n\geq 1}$ be a sequence of numbers in $[0,1]$.
We define $(Da)_n:=a_n-a_{n+1}$, $n\geq 1$. There exists a probability
measure $\lambda$ on $[0,1]$ such that $a_n=\int x^n\,\lambda(dx)$
if and only if
\[
(D^ka)_n \geq 0, \qquad \forall \, n\geq1,\;\forall k\geq 1.
\]
Moreover such $\lambda$ is unique.
\end{lemma}
\noindent
Remark that, setting $N=|I|$ and $M=|J|$, projectivity is equivalent to
\begin{equation}\label{pages2}
c^N_k = \sum_{\ell=0}^M c^{M+N}_{k+\ell} \, \binom{M}{\ell},
\qquad \forall \ 0\leq k\leq N.
\end{equation}
For $M=1$ we obtain
\begin{equation}\label{cc}
c^{N+1}_k+c^{N+1}_{k+1}=c^N_k, \qquad \forall \ 0\leq k\leq N.
\end{equation}
Let us set $m_n:=c^n_n$.
One proves easily by \eqref{cc} and recurrence on $k$ that
\[
(D^km)_n = c^{n+k}_n\in[0,1], \qquad \forall \ k,n\geq 1.
\]
Therefore $(m_n)_{n\geq1}$ defines a unique measure $\lambda$
satisfying \eqref{pages} for $k=0$. \eqref{pages} for general $k$
follows by induction from:
\[
 \begin{split}
   c_n^{n+k+1} & =c_n^{n+k}-c_{n+1}^{n+k+1}
          = \int_{[0,1]} \left[x^n(1-x)^k-x^{n+1}(1-x)^k \right] d\lambda
          \\ & = \int_{[0,1]} x^n(1-x)^{k+1} \, d\lambda.
\end{split}
\]
Thus $\lambda$ is the unique solution to the claim of
the Lemma. This uniqueness together with $c^n_k=
c^n_{n-k}$, implies that $\lambda$ is symmetric. Thus
any intricacy defines a measure as claimed.

We turn to the converse, considering a symmetric measure $\lambda$ on $[0,1]$
and defining $c$ by means of \eqref{exploo}. The coefficients depending only
on the cardinalities, $\I^c$ is trivially exchangeable.
The symmetry of $\lambda$ yields immediately $c^n_k=c^n_{n-k}$, and the normalization condition
is given by
\[
\sum_{k=0}^N \binom{N}{k} \, c^N_k = \int_{[0,1]} \sum_{k=0}^N \binom{N}{k} \, x^k(1-x)^{N-k}\, \lambda(dx)
= 1,
\]
i.e. $c\in\mathcal C(\N^*)$.
To prove the projectivity of $c$, namely \eqref{pages2}, we compute:
\[
\begin{split}
\sum_{\ell=0}^M c^{M+N}_{k+\ell} \, \binom{M}{\ell} & =
\int_{[0,1]} \left[ \sum_{\ell=0}^M \binom{M}{\ell} \, x^\ell(1-x)^{M-\ell} \right]
x^k(1-x)^{N-k} \lambda(dx)
\\ & = \int_{[0,1]} x^k(1-x)^{N-k} \lambda(dx) = c^N_k.
\end{split}
\]
Thus \eqref{pages2} and projectivity follow. The Lemma is proved.
\end{proof}

\begin{proof}[Proof of Proposition \ref{1917}]
First, let $\I^c$ be an intricacy. Lemma \ref{1917}
yields a symmetric probability measure $\lambda_c$ on $[0,1]$
satisfying \eqref{exploo}. If $W_c$ be a random variable with law $\lambda_c$,
then \eqref{eq:proj} is equivalent to \eqref{exploo}.

Conversely, suppose that $c=(c^I_S)_{S\subset I}$ has the form
\eqref{eq:proj} for some probability $\mathbb P$
defined by $W_c,Y_1,Y_2,\dots$ as in the statement.
Obviously $c^I_S\geq0$ and $\sum_{S\subset I} c^I_S=1$.
$c^I_S=c^I_{S^c}$ follows from the symmetry of $W_c$.
Thus $c$ is a system of coefficients.
Exchangeability of $c$ follows from exchangeability of the random
variables $\un{(Y_i\geq W_c)}$, $i\in I$.
By \eqref{eq:proj} we know that
\[
c^I_S = c^{|I|}_{|S|} = \E\left((1-W_c)^{|I\backslash S|} \, W_c^{|S|}\right) =
\int_{[0,1]} x^{|S|}(1-x)^{|I\backslash S|}\, \lambda_c(dx).
\]
Therefore, by Lemma \ref{1916} the functional $\I^c$ is an intricacy.

Let now $W_c,Y_1,Y_2,\dots$ be defined as in point 2 of the statement
and $\cZ$ defined by \eqref{explo}.
Each $i\in\N^*$ belongs to the random set $\cZ$ if and only if $Y_i\geq W_c$.
Conditionally on $W_c$, the probability of $\{Y_i\geq W_c\}$
is therefore $1-W_c$. As the variables $Y_1,Y_2,\dots$ are independent:
\[
\mathbb P(\cZ\cap I=S \, | \, W_c)=(1-W_c)^{|I\backslash S|} \, W_c^{|S|}.
\]
Averaging over the values of $W_c$ we obtain
\[
\mathbb P(\cZ\cap I=S)=\E\left((1-W_c)^{|I\backslash S|} \, W_c^{|S|}\right)
%= \int_{[0,1]} x^{|S|}(1-x)^{|I\backslash S|}\, \lambda_c(dx)
\]
and therefore \eqref{eq:proj} and \eqref{eq:proj2} are equivalent.
The last assertion follows from Lemma \ref{1916} and from \eqref{eq:proj2}.
The Proposition is proved.
\end{proof}

\subsection{Examples of intricacies}

We show that the Edelman-Sporns-Tononi neural complexity \eqref{est}
and the uniform and $p$-symmetric intricacies correspond
to natural probability laws on $[0,1]$. In particular, they are
weakly additive and really intricacies:
\begin{lemma}\label{1012}
In the setting of Lemma \ref{1916}
\begin{enumerate}
\item If $W_c$ is uniform on $[0,1]$ then $\I^c$ is the Edelman-Sporns-Tononi
neural complexity \eqref{est}.
\item If $W_c$ is uniform on $\{p,1-p\}$  then $\I^c$ is
the $p$-symmetric intricacy $\I^p$; in the case $p=1/2$,
$W_c=\frac12$ a.s. yields the uniform intricacy $\I^U$.
\end{enumerate}
\end{lemma}
\begin{proof}
Let $W_c$ be uniform on $[0,1]$. Then
\[
\begin{split}
\bbP(\z\cap I=\{1,\ldots,k\})
& = \bbP(Z_1=\cdots=Z_k=1, \,Z_{k+1}=\cdots=Z_N=0)
\\ & =\int_{[0,1]} x^k (1-x)^{N-k}\, dx =: a(k,N-k).
\end{split}
\]
We claim now that for all $k\geq 1$  and $j\geq 0$
\[
a(k,j) = \frac{j!}{(k+1)\cdots (k+j+1)} =
\frac1{(k+j+1) \, \binom{k+j}{k}},
\]
i.e., the Edelman-Sporns-Tononi coefficient $c^{k+j}_j$.

Indeed, for $j=0$ this reduces to $\int_0^1 x^k\, dx =1/(k+1)$.
To prove the general case, one fixes $k$ and uses recurrence on $j$. Indeed,
suppose we have the result for $j\geq 0$. Then
\[
\begin{split}
& \int_0^1 x^k(1-x)^{j+1}\, dx %= \int_0^1 p^k(1-p)^j (1-p)\, dp \\ &
= \int_0^1 x^k(1-x)^j\, dx - \int_0^1 x^{k+1}(1-x)^j\, dx
 \\ & = \frac1{(k+j+1)\binom{k+j}{k}}-\frac1{(k+j+2)\binom{k+j+1}{k+1}}
% \\ & = \frac{k!j!}{(k+j+1)!}\left(1-\frac{k+1}{k+j+2}\right)
% = \frac{k!j!}{(k+j+1)!} \, \frac{j+1}{k+j+2} \\ &
= \frac1{(k+j+2)\binom{k+j+1}{k}}.
\end{split}
\]
%Notice also that for all $k\geq 1$ %\annotation{why this notice?}
%\[
%\bbP(1,\ldots,k\in \z) = \bbP(Z_1=\cdots=Z_k=1) = a(k,0) = \frac1{k+1}.
%\]
If $W_c$ is uniform over $\{p,1-p\}$ then
\[
\int_{[0,1]} x^k(1-x)^{N-k}\, \frac12(\delta_{p}+\delta_{1-p})(dx) = \frac12(p^k(1-p)^{N-k}+(1-p)^kp^{N-k}),
\]
which is the coefficient $c^N_k$ of $\I^p$.
\end{proof}

\subsection{Further properties}
We deduce some useful facts from the above representation.

\begin{lemma}\label{alors}
The following are equivalent for an intricacy $\I^c$
with associated measure $\lambda_c$ as in Lemma \ref{1916}.
\begin{enumerate}
\item $\I^c$ is non-null, i.e. $c^N_k>0$ for at least one choice of
$N\geq 2$ and $1\leq k< N$;
\item  $c^N_k>0$ for all $N\geq 2$ and $1\leq k\leq N-1$;
\item $\lambda_c(]0,1[)>0$.
\end{enumerate}
\end{lemma}
\begin{proof}
We have:
\[
c^n_j = \int_{[0,1]} x^j(1-x)^{n-j}\, \lambda_c(dx)
\]
with $x^j(1-x)^{n-j}$ zero exactly at $x\in\{0,1\}$ whenever
$0<j<n$ and strictly positive on $]0,1[$. Thus
(1)$\implies$(3)$\implies$(2)$\implies$(1).
\end{proof}

\begin{lemma}\label{bert2}
If $\I^c$ is non-null, then $\I^c(X)=0$ for a $X\in\cX(d,N)$
if and only if $X=(X_1,\ldots,X_N)$ is an independent family.
\end{lemma}
\begin{proof}
It is enough to show that:
 $
   \I^c(X)=0\iff \H(X)=\sum_{i\in I} \H(X_i).
 $
If $\I^c$ is non-null and $\I^c(X)=0$, then by Lemma \ref{alors} we have
$\MI(S)=0$ for all $S\subset I$ with $S\notin\{\emptyset,I\}$. Therefore
 $\H(X)=\H(X_S)+\H(X_{S^c})$ and an easy induction yields the claim.
\end{proof}

\begin{example}[Systems with intricacy proportional to $N$]\label{gb5a}
{\rm We generalize the result of Example \ref{gb5} from $\I^U$
to a non-null intricacy $\I^c$. Considering the same system $X$ as in Example \ref{gb5}, we
get by Proposition \ref{1917}
 \[
 \begin{split}
    \frac{\I^c(X)}{\log d} & = \sum_{S\subset I} c^I_S \, |\Delta_S|
      = \mathbb E\left(|\Delta_{\cZ\cap I}|\right)
      %=\mathbb E\left(\sum_{k=1}^{N-1} \un{\left(\un{\cZ\cap I}(k)\ne
      %\un{\cZ\cap I}(k+1)\right)}\right)
      \\ & =\sum_{k=1}^{N-1} \mathbb P(\un{\cZ}(k)\ne \un{\cZ}(k+1))
      = (N-1) \, \mathbb P(\un{\cZ}(1)\ne \un{\cZ}(2)).
 \end{split}
 \]
By the probabilistic representation \eqref{eq:proj} through a random
variable $W_c$ with law $\lambda_c$ on $[0,1]$,
\begin{equation}\label{ppag}
 \kappa_c := \mathbb P(\un{\cZ}(1)\ne \un{\cZ}(2)) = \int_{[0,1]} 2x(1-x)\, \lambda_c(dx)
    \in\, ]0,1/2].
\end{equation}
Then we have obtained a system $X\in\cX(d^2,N)$ such that
\begin{equation}\label{clari}
 \I^c(X) = \frac{\kappa_c}2\, (N-1)\log d^2. \qquad\qquad \square
\end{equation}
 }
\end{example}

%%%%%%%%%%%%%%%%%%%%%%%%%%%%%%%%%%%%%%%%%%%%%%%%%%%%%%%%%%%%%%%%%%%%%%%%%%%%%%%%%%
%%%%%%%%%%%%%%%%%%%%%%%%%%%%%%%%%%%%%%%%%%%%%%%%%%%%%%%%%%%%%%%%%%%%%%%%%%%%%%%%%%
\section{Bounds for maximal intricacies}
%\subsection{Existence and upper bound for the mean intricacy}
\label{maximum}
%%%%%%%%%%%%%%%%%%%%%%%%%%%%%%%%%%%%%%%%%%%%%%%%%%%%%%%%%%%%%%%%%%%%%%%%%%%%%%%%%%
%%%%%%%%%%%%%%%%%%%%%%%%%%%%%%%%%%%%%%%%%%%%%%%%%%%%%%%%%%%%%%%%%%%%%%%%%%%%%%%%%%

In this section we prove Theorem \ref{thm:main}. We recall the
definition \eqref{ppag} for a non-null intricacy $\I^c$
\begin{equation}\label{ppagg}
\kappa_c=2\int_{[0,1]} x(1-x)\, \lambda_c(dx)= 2c^2_1>0.
\end{equation}
Recall that $\I^c(d,N)$ and $\I^c(d,N,x)$, defined in \eqref{supac0}
and \eqref{supac1}, denote the maximum of $\I^c$ over $\cM(d,N)$,
respectively over $\{\mu\in\cM(d,N): \H(\mu)=xN\log d\}$. We are
going to show the following
\begin{prop}\label{bubi}
Let $\I^c$ be a non-null intricacy and $d\geq2$. Then for all $N\geq 2$
\begin{equation}\label{true}
\frac{\kappa_c\log d}2 \left(1-\frac1N\right) \leq
\frac{\I^c(d,N)}N \leq \frac{\log d}2,
\end{equation}
and for any $x\in[0,1]$
\begin{equation}\label{truee}
\left[x\wedge(1-x)\right]\, \kappa_c\, \log d \left(1-\frac1N\right) \leq
\frac{\I^c(d,N,x)}N \leq \frac{1}2 \log d,
\end{equation}
where $\kappa_c>0$ is defined in \eqref{ppagg}.
\end{prop}
\begin{proof}
The upper bound for $\I^c(d,N)/N$ follows from Lemma \ref{bert}.
We show now the lower bound for $\I^c(d,N,x)/N$. Let $x\in\,]0,1[$.
In example \ref{gb3.5} we have constructed a system $X=(X_1,X_2)\in\cX(d,2)$ with
\[
\H(X)=x\log d^2, \qquad \I^c(X)=2\kappa_c\, [x\wedge(1-x)]\log d>0.
\]
Let now $(Y_{2i+1})_{i\geq 0}$ an i.i.d. family of copies of $X_1$
and set $Y_{2(i+1)}:=Y_{2i+1}$ for all $i\geq 0$.
Then, for $M\geq 1$, $Y:=(Y_i)_{i=1,\ldots,2M}\in\cX(d,2M)$
is the product of $M$ independent copies of $(X_1,X_2)$ and
by weak additivity
\[
\I^c(Y) = M \, \I^c(X) = 2M \, \kappa_c\, [x\wedge(1-x)]\log d, \qquad \H(Y)=2Mx\log d.
\]
If $S$ is a $\{0,\ldots,d-1\}$-valued random variable independent of $Y$ with $\H(Z)=x\log d$, then
$Z:=(Y_1,\ldots,Y_{2M},S)\in\cX(d,2M+1)$ satisfies by weak additivity
\[
\I^c(Z) = \I^c(Y_1,\ldots,Y_{2M}) =  2M \, \kappa_c\, [x\wedge(1-x)]\log d,\qquad \H(Z)=(2M+1)x\log d.
\]
Setting $N=2M$, respectively $N=2M+1$, we obtain the upper bound for $\I^c(d,N,x)/N$. Taking the
supremum over $x\in[0,1]$ in \eqref{truee}, we obtain \eqref{true}.
\end{proof}

\subsection{Super-additivity}
We are going to prove that the maps $N\mapsto\I^c(d,N)$ and $N\mapsto\I^c(d,N,x)$ are
super-additive.
By Lemma \ref{assumed}, the suprema defining $\I^c(d,N)$ and $\I^c(d,N,x)$
are maxima. The measures achieving the first supremum are called
{\it maximal intricacy measures}.
%Since the functional $\I^c$ is not concave, we do not expect uniqueness of maximizers.

%
\begin{lemma}
For any intricacy $\I^c$ and $d\geq 2$, the following limits exist. First,
\begin{equation}\label{omer}
    \I^c(d)= \lim_{N\to\infty} \frac{\I^c(d,N)}N = \sup_{N\geq 1} \frac{\I^c(d,N)}N  \in \ ]0,+\infty[
\end{equation}
and, for each $x\in \, ]0,1[$,
\begin{equation}\label{omerX}
       \I^c(d,x)=  \lim_{N\to\infty} \frac{\I^c(d,N,x)}N = \sup_{N\geq 1} \frac{\I^c(d,N,x)}N  \in \ ]0,+\infty[.
\end{equation}
\end{lemma}
\begin{proof}
We prove \eqref{omerX}, \eqref{omer} being similar and simpler. Fix $x\in \,]0,1[$.
For each $N\geq1$, let $a_N:=\I^c(d,N,x)$. We claim that this sequence
 is \emph{super-additive}, i.e.,
\[
a_{N+M} \geq a_N+a_M, \qquad \forall \ N,M\geq 1.
\]
Indeed, let $X^N$ and $X^M$ such that
\[
\begin{split}
\I^c(X^N)& =\I^c(d,N,x), \quad \H(X^N)=xN\log d,
\\ \I^c(X^M) & =\I^c(d,M,x), \quad \H(X^M)=xM\log d.
\end{split}
\]
Assume that $X^N$ and $X^M$ are independent. By weak-additivity
\[
\begin{split}
\I^c(X^N,X^M) & =\I^c(X^N)+\I^c(X^M),
\\ \H(X^N,X^M) & =\H(X^N)+\H(X^M)=x(N+M)\log d.
\end{split}
\]
Thus,
\[
\begin{split}
a_N+a_M & =\I^c(d,N,x)+\I^c(d,M,x)=\I^c\left(X^N\right)+\I^c\left(X^M\right)
\\ & =\I^c\left(X^N,X^M\right)\leq \I^c(d,N+M,x)=a_{N+M}.
\end{split}
\]
Moreover, by Proposition \ref{bubi}, we have
$\sup_{N\geq 1} a_N/N \leq (\log d)/2$. Therefore, by Fekete's Lemma
$a_N/N\to \sup_M a_M/M\leq (\log d)/2$ as $N\to+\infty$. Moreover, the limit
is positive by \eqref{truee}.
\end{proof}

\subsection{Adjusting Entropy}
To strengthen the previous result to obtain the second assertion
of Theorem \ref{thm:main},
we must adjust the entropy without significantly changing
the intricacy.

\begin{lemma}\label{lem:I-linear}
Let $X^{(1)},\dots,X^{(r)}\in\cX(d,N)$.
Let $U$ be a random variable over $\{1,\dots,r\}$,
independent of $\{X^{(1)},\dots,X^{(r)}\}$.
Let $Y:=X^{(U)}\in\cX(d,N)$, i.e., $Y=X^{(u)}$ whenever $U=u$. Then:
 \begin{equation}\label{eq:H-quasi-linear}
    0  \leq \H(Y_S) -    \sum_{u=1}^r \mathbb P(U=u) \, \H(X^{(u)}_S)
\leq \log r, \qquad \forall \, S\subset\{1,\ldots,N\},
 \end{equation}
 \begin{equation}\label{eq:H-quasi-linear-i}
    -\log r
      \leq \I^c(Y) -\sum_{u=1}^r \mathbb P(U=u) \, \I^c(X^{(u)}) \leq 2\log r.
 \end{equation}
\end{lemma}

\begin{proof}
We first prove \eqref{eq:H-quasi-linear}. By \eqref{coco},
 $$
    \H(Y_S|\,U)\leq \H(Y_S) \leq \H(Y_S,U)=\H(Y_S|\,U)+\H(U).
 $$
Now $\H(U)\leq \log r$. \eqref{eq:H-quasi-linear} follows as:
 $$
    \H(Y_S|\,U) = \sum_{u=1}^r \mathbb P(U=u) \H(Y_S|\,U=u)
     = \sum_{u=1}^r \mathbb P(U=u) \H(X^{(u)}_S).
 $$
\eqref{eq:H-quasi-linear-i} follows immediately, using \eqref{eq:I-as-H} and \eqref{eq:H-quasi-linear}.
\end{proof}

\begin{lemma}\label{lem:adjust-H}
Let $0<x<1$ and $\eps>0$ and $\I^c$ be some non-null intricacy. Then there exists $\delta_0>0$
and $N_0<\infty$ with the following property for all $0<\delta<\delta_0$ and $N\geq N_0$.
For any $X\in\cX(d,N)$ such that $\left|\frac{\H(X)}{N\log d}-x\right|\leq\delta$,
there exists $Y\in\cX(d,N)$ satisfying:
 $$
    \H(Y)=x N\log d, \qquad |\I^c(Y)-\I^c(X)|\leq \eps N \log d.
 $$
\end{lemma}

\begin{proof}
We fix $\delta_0=\delta_0(\eps,x)>0$ so small that:
 $$
    \frac{\delta_0}{\min\{1-x-\delta_0,x-\delta_0\}} < \eps/4
 $$
and $N_0=N_0(\eps,x,\delta_0)$ so large that:
 $$
    \frac{\log 2}{N_0\min\{1-x-\delta_0,x-\delta_0\}\log d} < \eps/4.
 $$
Let $N\geq N_0$ and $X\in\cX(d,N)$ be such that $\left|\frac{\H(X)}
{N\log d}-x\right|\leq\delta\leq\delta_0$. There are two similar
cases, depending on whether
$\H(X)$ is greater or less than $xN\log d$. We
assume $h:=\H(X)/N\log d<x$ and shall explain at the
end the necessary modifications for the other case.

Let $Z=(Z_i, i=1,\ldots,N)$ be i.i.d. random variables, uniform over
$\{0,\dots,d-1\}$. We consider $Y^t\in\cX(d,N)$ defined by
\[
Y^t:=X\, \un{(U\leq1-t)}+Z\, \un{(U>1-t)},
\]
where $U$ is a uniform random variable over $[0,1]$ independent of $X$ and $Z$.
$\I^c(Y^0)=\I^c(X)$ and $\I^c(Y^1)=\I^c(Z)=0$. Hence,
by the continuity of the intricacy, we get that there is
some $0<t_0<1$ such that $\H(Y^{t_0})=xN\log d$. Let us
check that $t_0$ is small.

By \eqref{eq:H-quasi-linear}
 \[
    0\leq \H(Y^t) - (1-t)\H(X)-t\H(Z) =
    \H(Y^t) - (1-t)hN\log d-tN\log d \leq \log 2.
    \]
so that, for some $\alpha\in[0,1]$,
 $$
    0<t_0=\frac{x-h}{1-h}- \frac{\alpha\log 2}{N(1-h)\log d}\leq \frac{\delta}{1-x-\delta}
      <\frac\eps2,
 $$
since $\delta\leq\delta_0$.
Thus, by \eqref{eq:H-quasi-linear-i}, setting $Y:=Y^{t_0}$,
\[
    \left|\I^c(Y)- (1-t_0)\I^c(X)-t_0\, \I^c(Z)\right|
    = \left|\I^c(Y)- (1-t_0)\I^c(X)\right| \leq 2\log 2,
\]
and therefore by \eqref{true}
\[
   \left|\I^c(Y)-\I^c(X)\right|\leq t_0\I^c(X)+ 2\log 2
      \leq \frac\eps2 \, N\log d+2\log 2.
\]
Dividing by $N\log d\geq N_0\log d$ we obtain the desired estimate.

For the case $h>x$, we use instead a system $Z$ with
constant variables, so that $\H(Z)=0=\I^c(Z)$ and a similar argument gives the result.
\end{proof}

\subsection{Proof of Theorem \ref{thm:main}}

Assertion (1) is already established: see Proposition \ref{bubi}.
It remains to complete the proof of the second assertion.

Let us set for $\delta\geq 0$
\[
\I^c(d,N,x,\delta):=\sup \left\{ \I^c(X): X\in\cX(d,N), \
            \left|\frac{\H(X)}{N\log d} - x\right| \leq \delta
            \right\}.
\]
We want to prove that
\[
\I^c(d,x) = \lim_{N\to+\infty}\frac1{N} \, \I^c(d,N,x,\delta_N).
\]
for any sequence $\delta_N\geq 0$ converging to $0$ as $N\to+\infty$.
We first observe that \eqref{omerX} gives that the limit exists and is equal
to $\I^c(d,x)$ if $\delta_N=0$, for all $N\geq1$. Consider now a general
sequence of non-negative numbers $\delta_N$ converging to zero.
Obviously, $\I^c(d,N,x,\delta_N)\geq \I^c(d,N,x,0)$, so that
\[
    \liminf_{N\to\infty} \frac1N\left(\I^c(d,N,x,\delta_N)
    -\I^c(d,N,x,0) \right)\geq 0.
    \]

Let us prove the reverse inequality for the $\limsup$. Let $\eps>0$.
Let $X^N\in\cX(d,N)$ realize $\I^c(d,N,x,\delta_N)$.
Let $\delta_0$ and $N_0$ be as in Lemma \ref{lem:adjust-H}.
We may assume that $N\geq N_0$ and $\delta_N<\delta_0$.
It follows that there is some $Y^N\in\cX(d,N)$ with the
entropy $Nx\log d$ such that $\I^c(Y^N)\geq \I^c(X^N)-\eps N$.
Hence, $\I^c(d,N,x,0)\geq \I^c(d,N,x,\delta_N)-\eps N$.
We obtain
\[
    \limsup_{N\to\infty} \frac1N\left(\I^c(d,N,x,\delta_n)
    -\I^c(d,N,x,0) \right)\leq\eps,
\]
Assertion (2) follows by letting $\eps\to0$. \quad \quad $\square$

%%%%%%%%%%%%%%%%%%%%%%%%%%%%%%%%%%%%%%%%%%%%%%%%%%%%%%%%%%%%%%%%%%%%%%%%%%%%%%%%%%
%%%%%%%%%%%%%%%%%%%%%%%%%%%%%%%%%%%%%%%%%%%%%%%%%%%%%%%%%%%%%%%%%%%%%%%%%%%%%%%%%%
%\section{Stationary and exchangeable systems}
\section{Exchangeable systems}
%%%%%%%%%%%%%%%%%%%%%%%%%%%%%%%%%%%%%%%%%%%%%%%%%%%%%%%%%%%%%%%%%%%%%%%%%%%%%%%%%%
%%%%%%%%%%%%%%%%%%%%%%%%%%%%%%%%%%%%%%%%%%%%%%%%%%%%%%%%%%%%%%%%%%%%%%%%%%%%%%%%%%

In this section we prove Theorem \ref{thm:exchsta}, % i.e., that
%stationarity does not constrain intricacy but exchangeability does.
%\subsection{Finite exchangeable families}
namely we prove that exchangeable systems have small intricacy. In particular,
one cannot approach the maximal intricacy with such systems.

\medbreak

\begin{prop}\label{prop:I-exch}
Let $\I^c$ be any mutual information functional and $d\geq2$. Then for all
$\gep>0$ there exists a constant $C=C(\gep,d)$  such that
for all exchangeable $X\in\cX(d,N)$
 \begin{equation}\label{eq:exex}
\I^c(X) \leq CN^{\frac23+\gep}, \qquad N\geq 2.
 \end{equation}
In particular
 $$
    \lim_{N\to\infty} \frac1N \max_{X\in\EX(d,N)} \I^c(X) = 0.
 $$
\end{prop}

\begin{proof}
Fix $\gep>0$.
Throughout the proof, we denote by $C$ constants which only depend on $d$ and $\gep$
and which may change value from line to line. Also $\kk=(k_1,\dots,k_d)\in\N^d$,
$\xx:=\frac1n\kk$ and $|\kk|:=k_1+\dots+k_d=n$ and
the multinomial coefficients and the entropy function are denoted by:
 $$
   \binom{n}{\kk} =
      \frac{n!}{k_1!k_2!\dots k_d!}, \qquad
   h(\xx) = -\sum_{i=1}^d x_i\log x_i.
 $$
We are going to use the following version of
Stirling's formula
\[
n! = \sqrt{2\pi n}\left(\frac ne\right)^n e^{\zeta_n}, \qquad
\frac1{12n+1}<\zeta_n<\frac1{12n}, \qquad n\geq 1.
\]
Therefore, for all $\kk\in\N^d$ such that $|\kk|= n$
 $$
   \binom{n}{\kk}=
     \left[ e^{n h(\xx)} (2\pi n)^{1/2}\prod_{x_i\ne0}(2\pi nx_i)^{-1/2}\right] g(\kk,n),
 $$
 where $ g(\kk,n) := \exp(\zeta_n-\zeta_{k_1}-\cdots-\zeta_{k_d})$ and therefore
 \[
\exp(-d)\leq g(\kk,n) \leq \exp(1).
 \]
In particular, as all non-zero $x_i$ satisfy $x_i\geq 1/n$,
 \begin{equation}\label{eq:combi-stirling}
   \left|\frac1n\log\binom{n}{\kk}- h(\xx)\right| \leq C\, \frac{\log n}n.
 \end{equation}
Let $X\in\EX(d,N)$. We set for $0\leq n\leq N$ and $|\kk|=n$
 $$
   p_{n,\kk} = \mathbb P(X_1=\dots=X_{k_1}=1,\dots,
     X_{k_1+\dots+k_{d-1}+1}=\dots=X_n=d).
 $$
These $\binom{n+d-1}{d-1}$ numbers determine
the law of any subsystem $X_S$ of size $|S|=n$.
It is convenient to define also $Y_i:=\#\{1\leq j\leq n:X_j=i\}$ for $i=0,\dots,d-1$ and
 $$
   q_{n,\kk} := \mathbb P(
      Y_i=k_i, \ i=0,\dots,d-1) = \binom{n}{\kk} \, p_{n,\kk}.
 $$
 Since the vector $(q_{n,\kk})_{|\kk|=n}$ gives the law of the vector $(Y_1,\ldots,Y_d)$ we have in particular
 \[
 \sum_{|\kk|=n} q_{n,\kk} = 1.
 \]
Second, we observe that for $|S|=n$
 \begin{equation}\label{eq:h-linear}
 \left|\frac{\H(X_S)}n-\frac1n\sum_{|\kk|=n} q_{n,\kk} \, h(\xx)\right|\leq C\, \frac{\log n}n.
  \end{equation}
  Indeed
  \[
   \begin{split}
   \frac{\H(X_S)}n & =  -\frac1n\sum_{|\kk|=n} q_{n,\kk} \, \log\frac{q_{n,\kk}}{\binom{n}{\kk}}
     = \sum_{|\kk|=n} q_{n,\kk} \, \frac1n\, \log\binom{n}{\kk}
       -\frac1n\sum_{|\kk|=n} q_{n,\kk} \log q_{n,\kk}\\
     & = \frac1n\sum_{|\kk|=n} q_{n,\kk} \, h(\xx)+G(n), \qquad |G(n)|\leq C\, \frac{\log n}n,
  \end{split}
  \]
  where we use \eqref{eq:combi-stirling} and the fact that
  \[
  -\sum_{|\kk|=n} q_{n,\kk} \log q_{n,\kk} =\H(Y_1,\ldots,Y_d) \leq d\log n,
  \]
since the support of the random vector $(Y_1,\ldots,Y_d)$ has cardinality at most $n^d$.

Third, we claim that, for $\gep>0$, there exists a constant $C$ such that for all
$N$ and all $X\in\EX(d,N)$, for all $n\in[\tilde N,N]$ with $\tilde N:=\lfloor N^{\frac23+\gep}+1\rfloor$,
 \begin{equation}\label{eq:h-propor}
     \left| \sum_{|\kk|=n} q_{n,\kk} h(\xx) - \sum_{|\KK|=N} q_{N,\KK} h(\XX) \right|
              \leq C\, N^{-\frac13+\gep},
 \end{equation}
where $\XX:=\frac1N\KK$ (no relation with the random variable $X$).
% and the implied constant independent of $\KK,\kk,n,N$.
By \eqref{eq:h-linear} and \eqref{eq:h-propor} we obtain for all $n\in[\tilde N,N]$ and $|S|=n$
 \begin{equation}\label{eq:h-ququ}
     \left| \frac{\H(X_S)}n - \frac{\H(X)}N \right|
              \leq C\, N^{-\frac13+\gep}.
 \end{equation}
Let us show how \eqref{eq:h-ququ} implies \eqref{eq:exex}. Using $\H(X_S)\leq\log d\cdot|S|$,
$\sum_{S\subset I} c^I_S=1$, we get
 $$
    \sum_{|S|< \tilde N} c^I_S\MI(S) \leq
    \sum_{S\subset I} c^I_S \times \log d\cdot \tilde N
    = \log d\cdot \tilde N.
 $$
Using \eqref{eq:I-as-H}, exchangeability of $X$,
$\sum_{n=0}^N c^N_n\binom{N}{n}=1$ and \eqref{eq:h-ququ}, we estimate
\[
 \begin{split}
  \I^c(X) & \leq 2 \cdot \log d \cdot \tilde N +
        2\sum_{n=\tilde N}^N \binom{N}{n} \, c^N_n \H(X_{\{1,\dots,n\}})
       -\H(X)\\
  & \leq 2\sum_{n=0}^N \binom{N}{n} \, c^N_n\, n \, \left(\frac{\H(X)}N + C\, N^{-\frac13+\gep}\right)
           -\H(X)  +C\tilde N.
 \end{split}
 \]
Finally, using $c^N_n\binom{N}{n}=c^N_{N-n}\binom{N}{N-n}$ and
$\sum_{n=0}^N c^N_n\binom{N}{n} =1$
\[
 \begin{split}
  \I^c(X)  & \leq \left(2\sum_{n=0}^N c^N_n\binom{N}{n} \frac{n}{N}-1\right)
      \H(X)+CN \times N^{-\frac13+\gep} + C\tilde N\\
   & \leq \left(\sum_{n=0}^N c^N_n\binom{N}{n} \left(\frac{n}{N}+\frac{N-n}{N}\right)-1\right) \H(X)+ CN^{\frac23+\gep}
   = CN^{\frac23+\gep}
  \end{split}
\]
and \eqref{eq:exex} is proved.

We turn now to the proof of \eqref{eq:h-propor}. We claim first that
\begin{equation}\label{kaka}
    p_{n,\kk} = \sum_{|\KK|=N, \, \KK\geq \kk} p_{N,\KK} \binom{N-n}{\KK-\kk}.
\end{equation}
Indeed, notice that
 $$
     p_{n,\kk}=\sum_{j=1}^d p_{n+1,\kk+\mathbf \delta^j}, \qquad
        \forall \ 0\leq n<N,\ \forall \ |\kk|=n,
 $$
where  $\delta^j:=(\delta^j_1,\dots,\delta^j_d)$ with
$\mathbf\delta^j_i=1$ if $i=j$, $0$ otherwise. This in particular yields
\eqref{kaka} for $N=n+1$. Moreover if $|\KK|=n+1$ then
 $$
   \binom{n+1}{\KK}=\sum_{j=1}^d \binom{n}{\KK-\mathbf\delta^j} \, \un{(\KK\geq \delta^j)}.
 $$
 Then, arguing by induction on $N\geq n$
 \[
 \begin{split}
 p_{n,\kk} & = \sum_{|\KK|=N, \, \KK\geq \kk} p_{N,\KK} \binom{N-n}{\KK-\kk}
 = \sum_{|\KK|=N, \, \KK\geq \kk} \sum_{j=1}^d p_{N+1,\KK+\mathbf \delta^j} \binom{N-n}{\KK-\kk}
 \\ & = \sum_{|\KK'|=N+1} p_{N+1,\KK'} \sum_{j=1}^d \binom{N-n}{\KK'-\kk-\mathbf \delta^j}
 \, \un{(\KK-\kk\geq \delta^j)}
 \\ & = \sum_{|\KK'|=N+1, \KK'\geq\kk} p_{N+1,\KK'} \binom{N+1-n}{\KK'-\kk}.
 \end{split}
 \]
We recall that
$q_{n,\kk} = \binom{n}{\kk} \, p_{n,\kk}$.
Notice that it is enough to prove claim \eqref{eq:h-propor} in
the case $q_{N,\kk'}=\delta_{\kk',\KK}$, i.e., $p_{N,\kk'}=\binom{N}{\kk'}^{-1}$ for
$\kk'=\KK$ and zero otherwise, if we find a constant $C$ which does not depend on
$(N,n,\KK)$.
Indeed, the two expressions are linear
and the average of $CN^{-1/3+\gep}$ will remain of the same
order. Thus, we need to estimate:
 $$
   a(N,\KK,n,\kk) :=
   q_{n,\kk}
   =  \binom{n}{\kk} \times \binom{N}{\KK}^{-1} \binom{N-n}{\KK-\kk}.
 $$
Let $\xx:=\kk/n\in[0,1]^d$, $\XX:=\KK/N\in[0,1]^d$ and $\nu=:n/(N-n)$.
Formula \eqref{eq:combi-stirling} implies that $\frac1n\log a(N,\KK,n,\kk)$ is equal to:
 $$
      \underbrace{
      h(\xx)  -(1+\nu^{-1}) h(\XX) + \nu^{-1}h(\XX+\nu(\XX-\xx))}_{
          =: \phi_{\nu,\XX}(\xx)}+G(N,n),
 $$
where $|G(N,n)|\leq \kappa(\log N)/n$, for some $\kappa=\kappa(d)$.

Let us now write for all $(x_1,\ldots,x_{d-1})\in[0,1]^{d-1}$ such that $\sum_i x_i\leq 1$
\[
H(x_1,\ldots,x_{d-1}):=h(x_1,\ldots,x_{d}), \quad x_d:=1-x_1-\ldots-x_{d-1}.
\]
Observe that for $i,j\leq d-1$
\[
\frac{\partial H}{\partial x_i}=\log\left(\frac{x_d}{x_i}\right),
\qquad \frac{\partial^2 H}{\partial x_i\partial x_j} = -\frac1{x_d}
- \frac1{x_i}\, \un{(i=j)}.
\]
In particular the Hessian of $H$ is negative-definite, since for all $a\in\R^{d-1}\backslash\{0\}$
\[
\sum_{i,j=1}^{d-1} a_ia_j \frac{\partial^2 H}{\partial x_i\partial x_j} =
-\frac1{x_d}\left(\sum_{i=1}^{d-1} a_i\right)^2- \sum_{i=1}^{d-1} \frac1{x_i}\, a_i^2\leq
- \sum_{i=1}^{d-1} a_i^2
\]
where we use the fact that $x_i\leq 1$. Hence, $h$ is concave and we obtain
%$\phi_{\nu,\XX}$ is a strictly concave function defined on a convex set and therefore it has a unique maximum.
%Again by concavity of $H$
\[
\phi_{\nu,\XX}(\xx)=\frac{\nu+1}\nu\left[\frac\nu{\nu+1}\, h(\xx)+\frac1{\nu+1}\,
h((1+\nu)\XX-\nu \xx)-h(\XX)\right]\leq 0,
\]
so that the maximum of $\phi_{\nu,\XX}$ is $0=\phi_{\nu,\XX}(\XX)$. The second order derivative estimate gives:
 $$
    \phi_{\nu,\XX}(\xx)\leq -2 \|\xx-\XX\|^2
     \qquad \text{ where }\|\xx\|:=\sqrt{x_1^2+\dots+x_d^2}.
 $$
Combining with the bound $|G(N,n)|\leq \kappa(\log N)/n$ above, we get, for all
$n<N$:
 $$
   a(N,\KK,n,\kk)\leq N^\kappa \times e^{-2n\|\xx-\XX\|^2}.
 $$
Recall $n\geq\tilde N=N^{\frac23+\gep}$ and set
$\delta:=N^{-\frac13}$ and
 $$
   \omega:=\sup_{\|\xx-\XX\|<\delta} \|h(\XX)-h(\xx)\|
        \leq C \, \delta \log \frac1{\delta}.
 $$
Finally, using $h(\xx)\leq\log d$,
\[
 \begin{split}
&  \left|\sum_{|\kk|=n} q_{n,\kk} h(\xx) - h(\XX)\right|
    \leq \omega \sum_{\|\xx-\XX\|<\delta} q_{n,\kk}
      + 2\log d  \sum_{\|\xx-\XX\|\geq \delta} q_{n,\kk} \\
 &   \leq C \,\delta \log \frac1{\delta}+ C \, n^d\, N^\kappa \,
     e^{-2\tilde N\delta^2}
    \leq C (\log N) N^{-\frac13} +
      C N^{\kappa+d} e^{-2N^\gep} \leq C N^{-\frac13+\gep}.
 \end{split}
 \]
Then \eqref{eq:h-propor} and the proposition are proved.
\end{proof}

\section{Small support}

In this section we prove Theorem \ref{prop:support}, namely
we show that exact
maximizers have small support. Numerical experiments suggest
that this support has in fact cardinality of order $d^{N/2}$.
We are able to prove the following weaker estimate.
For a fixed law $\mu\in\cM(d,N)$, we call forbidden configurations
the elements of $\Lambda_{d,N}:=\{0,\dots,d-1\}^N$ with zero $\mu$-probability.

\begin{prop}\label{prop:forbidden}
Let $\I^c(X)$ be a non-null intricacy.
Let $d=2$ and $N$ large enough. Let $\mu\in\cX(d,N)$ be a maximizer of $\I^c$. The forbidden
configurations are a lower-bounded fraction of all configurations:
 $$
   \#\{\omega\in\Lambda_{d,N}:\mu([\omega])=0\}\geq c(d) |\Lambda_{d,N}|,
 $$
for some $c(d)>0$ independent of $N$.
\end{prop}

\begin{proof}
If $\I^c$ is non-null, then $\lambda_c(\{0,1\})=2\lambda_c(\{0\})<1$ and
therefore $\lambda_c(\{0\})<1/2$. However we can without loss of generality
suppose that $\lambda_c(\{0\})=0$: indeed it is enough to remark that
\begin{enumerate}
\item the probability measure $\lambda_0:=\frac{\delta_0+\delta_1}2$ is associated
with the null intricacy $\I^0\equiv 0$,
\item the correspondence $\lambda_c\mapsto \I^c$ is linear and one-to-one,
\item we can write $\lambda_{c}=\alpha \lambda_0+(1-\alpha)\lambda_{c'}$, where
\[
\alpha:=2\lambda_c(\{0\})<1, \qquad \lambda_{c'}([a,b])= \frac{\lambda_{c}([a,b]\,
\cap \, ]0,1[)}{\lambda_{c}(]0,1[)},
\quad \forall \, a\leq b.
%\lambda_{c'}(dx):= \lambda_{c}(dx \,| \ ]0,1[).
%\qquad \lambda_{c'}(dx):= \frac{\un{]0,1[}(x)}{\lambda_{c}(]0,1[)}\lambda_{c}(dx).
\]
\end{enumerate}
Therefore $\I^c=\alpha\I^0+(1-\alpha)\I^{c'}=(1-\alpha)\I^{c'}$ and $\I^{c'}$
has the same maximizers as $\I^c$ but with $\lambda_{c'}(\{0\})=0$

We fix some large integer $z$ (how large will be explained below),
$N>z$ and $d\geq2$ and we consider the intricacy $\I^c$ as a function defined
on the simplex $\cM(d,N)=\{(p_\omega)_{\omega\in\Lambda_{d,N}} \in\mathbb R_+^{d^N}:\sum_{\omega\in\Lambda_{d,N}} p_\omega=1\}$.
A straightforward computation yields:
 $$
   \frac{\partial \I^c}{\partial p_\omega} = -2\sum_{S\subset I} c^I_S \log
   \left(\sum_{\alpha\equiv\omega[S]} p_\alpha \right)
     + \log p_\omega-1
 $$
where $\alpha\equiv\omega [S]$ iff $\alpha_i=\omega_i$ for all $i\in S$.
The second derivatives are:
 $$
   \frac{\partial^2 \I^c}{\partial p_\omega^2} =
    -2\sum_{S\subset I}  \frac{c^I_S}{\sum_{\alpha\equiv\omega[S]} p_\alpha}
     + \frac1{p_\omega},
     \quad  \frac{\partial^2 \I^c}{\partial p_{\omega_0} \partial p_{\omega_1}} =
    -2\sum_{S\subset I}  \frac{c^I_S}{\sum_{\alpha\equiv\omega_0[S]} p_\alpha}
    \, \un{(\omega_0=\omega_1[S])},
 $$
for $\omega_0\ne\omega_1$.

Let $p=(p_\omega)_{\omega\in\Lambda_{d,N}}$ be a maximizer of $\I^c$.
We show that for each $\beta\in\{0,\dots,d-1\}^{N-z}$,
 $$
   \Omega_\beta:=\left\{(\alpha_1,\ldots,\alpha_z,\beta_1,\ldots,\beta_{N-z})\in
   \{0,\dots,d-1\}^N: \alpha\in\{0,\dots,d-1\}^z\right\}
 $$
must contain at least one configuration forbidden by $p$. The claim will follow since
the cardinality of $\{0,\dots,d-1\}^{N-z}$ is $d^N/d^z$.

We assume by
contradiction the existence of some $\beta\in\{0,\dots,d-1\}^{N-z}$ such that
no configuration in $\Omega_\beta$ is forbidden. Let $\omega_0\in\Omega_\beta$
be such that
 $$
   p_{\omega_0}:=\min\{p_{\omega}:\omega\in\Omega_\beta\}>0.
 $$
Let now $\omega_1\in\Omega_\beta\setminus\{\omega_0\}$, which exists since
$|\Omega_\beta|\geq d\geq 2$, so that $p_{\omega_1}\geq p_{\omega_0}>0$.
We set for $t\in\,]-\gep,\gep[$ and $0<\gep<p_{\omega_0}$
\[
p_\omega^t:=\left\{ \begin{array}{ll}
p_{\omega_1}+t, \ & \omega=\omega_1, \\
p_{\omega_0}-t, \ & \omega=\omega_0, \\
p_\omega, & \omega\notin\{\omega_0,\omega_1\}.
\end{array} \right.
\]
Then $p^t$ is still a probability measure for $t\in\,]-\gep,\gep[$,
since $p_{\omega_1}\geq p_{\omega_0}>\gep>0$.

Since $p$ is a maximizer, then $\varphi(t):=\I^c(p^t)\leq
\varphi(0):=\I^c(p)$ for $t\in\,]-\gep,\gep[$. Then
\[
\begin{split}
0& \geq \varphi''(0) =\frac{\partial^2 \I^c}{\partial p_{\omega_0}^2}
+ \frac{\partial^2 \I^c}{\partial p_{\omega_1}^2} - 2
\frac{\partial^2 \I^c}{\partial p_{\omega_0} \partial p_{\omega_1}}
\\ & = \frac1{p_{\omega_1}}+\frac1{p_{\omega_0}}-2\sum_{S\subset I}  \un{(\omega_0\ne\omega_1[S])}
\left[\frac{c^I_S}{\sum_{\alpha\in[\omega_0]_S} p_\alpha} + \frac{c^I_S}{\sum_{\alpha\in[\omega_1]_S} p_\alpha}
\right]
\end{split}
\]
where $[\omega]_S=\{\alpha: \alpha=\omega \mod [S]\}$ is the equivalence
class of $\omega$. Therefore
 $$
0 \geq \frac1{p_{\omega_1}}+\frac1{p_{\omega_0}}\left(1-
        2\sum_{S\subset I} \frac{c^I_S}{|[\omega_0]_S\cap\Omega_\beta|}-
        2\sum_{S\subset I} \frac{c^I_S}{|[\omega_1]_S\cap\Omega_\beta|}
    \right)
 $$
and for some $\omega\in\Omega_\beta$
 \begin{equation}\label{eq:12}
    \sum_{S\subset I} \frac{c^I_S}{|[\omega]_S\cap\Omega_\beta|} > \frac14.
 \end{equation}
On the other hand,
%setting $k:=|S|$ and $s:=|S\cap\{1,\dots,z\}|$,
we have:
 $$
    |[\omega]_S\cap\Omega_\beta| = d^{|S^c\cap\{1,\dots,z\}|}
 $$
so that by Proposition \ref{1917}
the left hand side of \eqref{eq:12} is equal to:
\[
\begin{split}
\E\left( d^{-|\cZ^c\cap\{1,\dots,z\}|}\right) & = \int_{[0,1]}
\lambda_c(dx) \, \E\left( \prod_{i=1}^z d^{-\un{(Y_i<x)}}\right)
=\int_{[0,1]} \lambda_c(dx) \, \left(  \frac xd+(1-x)\right)^z
\\ & = \lambda_c(\{0\})+\int_{]0,1]}
\lambda_c(dx) \, \left(  \frac xd+(1-x)\right)^z.
\end{split}
\]
%By the symmetry of $\lambda_c$
%\[
%E\left( d^{-|\cZ^c\cap\{1,\dots,z\}|}\right) =
%\int_0^{1/2} \lambda_c(dx) \left[ \left(  \frac xd+(1-x)\right)^z
%+ \left(  \frac {1-x}d+x\right)^z \right]\leq 2^{-z-1}
%\]
Since we have reduced above to the case $\lambda_c(\{0\})=0$, then the latter expression tends to $0$ as $z\to+\infty$, contradicting \eqref{eq:12}.
\end{proof}

\appendix

%%%%%%%%%%%%%%%%%%%%%%%%%%%%%%%%%%%%%%%%%%%%%%%%%%%%%%%%%%%%%%%%%%%%%%%%%%%%%%%%%%
%%%%%%%%%%%%%%%%%%%%%%%%%%%%%%%%%%%%%%%%%%%%%%%%%%%%%%%%%%%%%%%%%%%%%%%%%%%%%%%%%%
\section{Entropy}
%%%%%%%%%%%%%%%%%%%%%%%%%%%%%%%%%%%%%%%%%%%%%%%%%%%%%%%%%%%%%%%%%%%%%%%%%%%%%%%%%%
%%%%%%%%%%%%%%%%%%%%%%%%%%%%%%%%%%%%%%%%%%%%%%%%%%%%%%%%%%%%%%%%%%%%%%%%%%%%%%%%%%

In this Appendix, we recall needed facts from basic information
theory. The main object is the entropy functional which may be
said to quantify the randomness of a random variable.

Let $X$ be a random variable taking values in a finite space $E$.
We define the {\it entropy} of $X$
\[
\H(X) := -\sum_{x\in E} P_X(x) \, \log(P_X(x)), \qquad
P_X(x):=\bbP(X=x),
\]
where we adopt the convention
\[
0\cdot \log(0) = 0\cdot \log(+\infty)=0.
\]
We recall that
\begin{equation}\label{orly}
0\leq \H(X) \leq \log |E|,
\end{equation}
More precisely, $\H(X)$ is minimal iff $X$ is a constant,
it is maximal iff $X$ is uniform over $E$. To prove \eqref{orly}, just notice that
since $\varphi\geq 0$ and $\varphi(x)=0$ if and only if $x\in\{0,1\}$, and
by strict convexity of $x\mapsto \varphi(x)=x\log x$ and Jensen's inequality
\[
\begin{split}
\log |E|-\H(X) & = \frac1{|E|}\sum_{x\in E} P_X(x)\, |E| \left(\log(P_X(x))+ \log |E| \right)
\\ & = \frac1{|E|}\sum_{x\in E} \varphi\left(P_X(x)\, |E|\right) \geq \varphi
\left(\frac1{|E|}\sum_{x\in E} P_X(x)\, |E|\right) = \varphi(1)=0,
\end{split}
\]
with $\log |E|-\H(X)=0$ if and only if $P_X(x)\, |E|$ is constant in $x\in E$.

\smallskip
If we have a $E$-valued random variable $X$ and a $F$-valued
random variable $Y$ defined on the same probability space,
with $E$ and $F$ finite, we can consider the vector
$(X,Y)$ as a $E\times F$-valued random variable The entropy of $(X,Y)$ is
then
\[
\H(X,Y) := -\sum_{x,y} P_{(X,Y)}(x,y) \, \log(P_{(X,Y)}(x,y)), \quad
P_{(X,Y)}(x,y):=\bbP(X=x,Y=y).
\]
This entropy $\H(X,Y)$ is a measure of the extent to which the "randomness of the
two variables is shared". The following notions formalize this
idea.

\subsection{Condidional Entropy}
The {\it conditional entropy} of $X$ given $Y$ is:
\[
\H(X\, |\, Y) := \H(X,Y) - \H(Y).
\]
We claim that
\begin{equation}\label{coco}
0\leq \H(X\, |\, Y) \leq \H(X)\leq \H(X,Y).
\end{equation}
Remark that $P_X(x)$ and $P_Y(y)$, defined in the obvious way,
are the marginals of $P_{(X,Y)}(x,y)$, i.e.
\[
P_X(x) = \sum_y P_{(X,Y)}(x,y), \qquad
P_Y(y) = \sum_x P_{(X,Y)}(x,y).
\]
In particular, $P_X(x)\geq P_{(X,Y)}(x,y)$ for all $x,y$. Therefore
\[
\sum_{x,y} P_{(X,Y)}(x,y) \, \log\left(\frac{P_{(X,Y)}(x,y)}{P_X(x)}\right)
\leq 0
\]
which yields
\[
\H(X,Y)=- \sum_{x,y} P_{(X,Y)}(x,y) \, \log P_{(X,Y)}(x,y)
\geq -\sum_x  P_X(x) \, \log P_X(x) = \H(X),
\]
i.e. $\H(X,Y)\geq \H(X)$ and $\H(X|Y)\geq0$. Therefore
\begin{equation}\label{geq}
\H(X,Y)\geq \max \{\H(X),\H(Y)\}.
\end{equation}
Moreover $\H(X,Y)=\H(X)$, i.e. $\H(Y|X)=0$, if and only if
$P_{(X,Y)}(x,y)=P_X(x)$ whenever
$P_{(X,Y)}(x,y)\ne 0$, i.e. $Y$ is a function of $X$.
On the other hand,
 \begin{equation}\label{leq}
       \H(X,Y) \leq \H(X)+\H(Y)
 \end{equation}
with equality, i.e., $\H(Y|X)=\H(Y)$, if and only if $X$ and $Y$ are
independent. This shows that $\H(X\, |\, Y) \leq \H(X)$ and completes the proof of \eqref{coco}.
Formula \eqref{leq} can be shown by considering
the Kullback-Leibler divergence or relative entropy:
\[
I := \sum_{x,y} P_{(X,Y)}(x,y) \, \log\left(\frac{P_{(X,Y)}(x,y)}{P_X(x)\,P_Y(y)}\right).
\]
Since $\log(\cdot)$ is concave, by Jensen's inequality
\[
-I \leq \log \left( \sum_{x,y} P_{(X,Y)}(x,y) \,
\frac{P_X(x)\,P_Y(y)}{P_{(X,Y)}(x,y)}\right) =
\log \left( \sum_{x,y}P_X(x)\,P_Y(y)\right) = 0.
\]
By strict concavity, $I=0$ if and only if $P_{(X,Y)}(x,y)=P_X(x)\,P_Y(y)$
for all $x,y$, i.e., whenever $X$ and $Y$ are independent.

By the above considerations, $\H(X\, |\, Y)\in[0,\H(X)]$ is
a measure of the uncertainty associated with $X$ if
$Y$ is known. It is minimal iff $X$ is a function of $Y$ and
it maximal iff $X$ and $Y$ are independent.

\subsection{Adding information decreases uncertainty}
Let us consider three random variables $(X,Y,Z)\mapsto E\times F\times G$
with $E,F,G$ finite. Then we have that
\begin{equation}\label{birge}
\H(X \, | \, (Y,Z) )\leq \H(X\, |\, Y).
\end{equation}
Indeed, this is equivalent to
\[
\H(X,Y,Z)+\H(Y) \leq \H(X,Y)+\H(Y,Z).
\]
Consider the quantity
\[
J := \sum_{x,y,z} P_{(X,Y,Z)}(x,y,z) \,
\log\left(\frac{P_{(X,Y,Z)}(x,y,z)\, P_Y(y)}{P_{(X,Y)}(x,y) \, P_{(Y,Z)}(y,z)}\right).
\]
Since $-\log(\cdot)$ is convex, by Jensen's inequality
\[
J \geq -\log \left( \sum_{x,y} \frac{P_{(X,Y)}(x,y) \, \sum_z P_{(Y,Z)}(y,z)}{P_Y(y)}
\right) = -\log \left( \sum_{x,y}P_{(X,Y)}(x,y)\right) = 0,
\]
and the inequality follows.

\subsection{Mutual Information}\label{mi}
Finally, we recall the notion of
{\it mutual information} between two random variables $X$ and $Y$
defined on the same probability space:
\[
\begin{split}
\MI(X,Y) & := \H(X)+\H(Y)-\H(X,Y)
\\ &  = \H(X)-\H(X\, |\, Y) = \H(Y)-\H(Y\, |\, X)
\\ & = \sum_{x,y} P_{(X,Y)}(x,y) \, \log\left(\frac{P_{(X,Y)}(x,y)}{P_X(x)\,P_Y(y)}\right).
\end{split}
\]
This quantity is a measure of the common randomness of $X$ and $Y$.
By \eqref{geq} and \eqref{leq} we have
$\MI(X,Y)\in[0,\min\{\H(X),\H(Y)\}]$. $\MI(X,Y)$ is
minimal (zero) iff $X,Y$ are independent and maximal, i.e. equal
to $\min\{\H(X),\H(Y)\}$, iff one variable is a function of
the other.
%, while it is maximal

Mutual information is non-decreasing. Let $X,X',Y,Y',\hat X,\hat Y$
be random variables such that $X,X'$, resp. $Y,Y'$, are
(deterministic) functions of $\hat X$, resp. $\hat Y$. Then:
 \begin{equation}\label{eq:MI-mono}
    \MI(X,Y)\leq \MI(\hat X,\hat Y).
 \end{equation}
Mutual information is almost additive:
 \begin{equation}\label{eq:MI-add}
   \left| \MI((X,Y),(X',Y')) - (\MI(X,X')+\MI(Y,Y')) \right|
     \leq \MI(\hat X,\hat Y).
 \end{equation}
These properties follow from the properties of conditional entropy. First,
 \begin{multline*}
   \MI(\hat X,\hat Y)=\H(\hat X)+\H(\hat Y)-\H(\hat X,\hat Y) \\
    = \H(X)+\H(\hat X|X)+\H(Y)+\H(\hat Y|Y)
      -\H(X,Y)-\H(\hat X|X,Y)-\H(\hat Y|\hat X,Y) \\
    =\MI(X,Y)
     + (\H(\hat X|X)-\H(\hat X|X,Y))
     + (\H(\hat Y|Y)-\H(\hat Y|\hat X,Y)).
 \end{multline*}
\eqref{eq:MI-mono} now follows from \eqref{birge}.
Second,
 \begin{multline*}
   \MI((X,Y),(X',Y'))=\H(X,Y)+\H(X',Y')-\H(X,X',Y,Y') \\
    = \H(X)+\H(Y)-\MI(X,Y)+\H(X')+\H(Y')-\MI(X',Y')\\
     \qquad -\H(X,X')-\H(Y,Y')+\MI((X,X'),(Y,Y')) \\
    =\H(X)+\H(X')-\H(X,X')+\H(Y)+\H(Y')-\H(Y,Y') \\
      \qquad +(\MI((X,X'),(Y,Y'))-\MI(X,Y)-\MI(X',Y')) \\
    =\MI(X,X')+\MI(Y,Y')
      + (\MI((X,X'),(Y,Y'))-\MI(X,Y)-\MI(X',Y')).
 \end{multline*}
The nonnegativity of mutual information and \eqref{eq:MI-mono}
yields
 \begin{multline*}
     -\min(\MI(X,Y),\MI(X',Y'))\leq
        \MI((X,Y),(X',Y')) - (\MI(X,X')+\MI(Y,Y')) \\
     \leq \MI((X,X'),(Y,Y')).
 \end{multline*}
\eqref{eq:MI-add} follows.


\begin{thebibliography}{99999}

\bibitem{Bak95}
P. Bak, M. Paczuski, Complexity, contingency and criticality,
{\it Proc. Natl. Acad. Sci. USA} {\bf 92} (1995), 6689-6696.

\bibitem{Barnett09}
L. Barnett, C. L. Buckley, S. Bullock, Neural complexity and structural connectivity,
{\it Phys. Rev. E} {\bf 79} (2009), 051914.

\bibitem{Bennett90}
C. Bennett, How to define complexity in physics and why, in: {\it Complexity,
entropy and the physics of information,} vol. VIII, eq. W. Zurek, Addison-Wesley (1990).

\bibitem{bertoin} J. Bertoin, {\it Random Fragmentation And Coagulation Processes},
Cambridge University Press (2006).

\bibitem{BZ2}
J. Buzzi, L. Zambotti, {\it Approximate Maximizers of Intricacy
Functionals,} preprint (2009), http://arxiv.org/abs/0909.2120.

\bibitem{InfoTheory}
T. Cover, J. Thomas, {\it Elements of Information Theory,} John Wiley \& Sons,
2006.

\bibitem{Crutchfield}
J. Crutchfield, K. Young, Inferring statistical complexity, {\it Phys. Rev. Lett.} {\bf 63} (1989), 105--109. 

\bibitem{DeLuca04}
M. De Lucia, M. Bottaccio, M. Montuori, L. Pietronero, A topological approach to neural complexity, {\it Phys. Rev. E} {\bf 71} (2005), 016114 --- arXiv:nlin/0411011v1.

\bibitem{moment}
W. Feller, {\it An introduction to probability theory and its applications.
Vol. 2,} John Wiley \& Sons, 1971.

\bibitem{Edelman01}
G. Edelman, J. Gally, Degeneracy and complexity in biological systems, {\it Proc. Natl. Acad. Sci. USA,} {\bf 98} (2001), 13763--13768.

\bibitem{Goldenfeld99}
N. Goldenfeld, L. Kadanoff, Simple lessons from complexity, Science {\bf 284} (1999), 87--89.

\bibitem{Entropy}
A. Greven, G. Keller, G. Warnecke, {\it Entropy,} Princeton University Press, 2003.

\bibitem{Holthausen99}
K. Holthausen, O. Breidbach, Analytical description of the evolution
of neural networks: learning rules and complexity, Biol. Cybern. {\bf 81} (1999), 169--176.

\bibitem{Krichmar04}
J. Krichmar, D. Nitz, J. Gally, G. Edelman, Characterizing functional hippocampal pathways in a brain-based device as it solves a spatial memory task, {\it Proc. Natl. Acad. Sci. USA,} {\bf 102} (2005), 2111--2116.

\bibitem{Seth06}
A. Seth, E. Izhikevich, G. Reeke, G. Edelman, Theories and measures of consciousness: an extended framework, {\it Proc. Natl. Acad. Sci. USA,} {\bf 103} (2006), 10799--10804.

\bibitem{Seth07}
Anil K Seth (2007), Models of consciousness, Scholarpedia, 2(1):1328.

\bibitem{Shanahan08}
M. P. Shanahan, Dynamical complexity in small-world networks of spiking neurons, Phys. Rev. E 78, 041924 (2008).

\bibitem{Sporns00}
O. Sporns, G. Tononi, G. Edelman, Connectivity and complexity: the relationship between neuroanatomy and brain dynamics, Neural Netw. 2000 Oct-Nov;13(8-9):909-22.

\bibitem{Sporns02}
O. Sporns, Networks analysis, complexity, and brain function, {\it Complexity} {\bf 8} (2002), 56 - 60.

\bibitem{Sporns07}
O. Sporns (2007), Complexity, Scholarpedia, 2(10):1623.

\bibitem{Tala}
M. Talagrand, {\it Spin Glasses: A Challenge for Mathematicians,} Springer, 2003.

\bibitem{Tononi94}
G. Tononi, O. Sporns, G. Edelman, A measure for brain complexity: relating
functional segregation and integration in the nervous system, {\it Proc. Natl. Acad. Sci. USA,} {\bf 91} (1994), 5033--5037.

\bibitem{Tononi96}
G. Tononi, O. Sporns, G. Edelman, A complexity measure for selective matching of signals by the brain, {\it Proc. Natl. Acad. Sci. USA,} {\bf 93} (1996), 3422--3427.

\bibitem{Tononi99}
G. Tononi, O. Sporns, G. Edelman, Measures of degeneracy and redundancy in biological networks, {\it Proc. Natl. Acad. Sci. USA,} {\bf 96} (1999), 3257--3262.

\end{thebibliography}
\end{document}